\documentclass[11pt]{article}
\usepackage{hyperref}

\usepackage[toc,page]{appendix}
\usepackage[utf8]{inputenc}
\usepackage{amsmath, latexsym, stmaryrd, amsthm, caption,amssymb
	,xcolor,subfig,amssymb,bbm, multirow, multicol, makecell,graphics,pgfplots,bm,
	,tabularx,a4wide,xurl,breakcites,enumitem,mathtools
}
\usepackage[nameinlink]{cleveref}
\hypersetup{
	colorlinks=true,
	linkcolor=blue,         
	citecolor=blue,         
	urlcolor=blue,
	pdfstartview=FitV
}

\numberwithin{equation}{section}
\usepackage[T1]{fontenc}
\usepackage[normalem]{ulem}
\newcolumntype{Y}{>{\raggedleft\arraybackslash}X}

\newcommand{\N}{\mathbb N}
\newcommand{\R}{\mathbb R}
\newcommand{\Z}{\mathbb Z}
\newcommand{\E}{\mathbb E}

\renewcommand{\S}{\mathbb S}

\renewcommand{\P}{\mathbb P}

\newcommand{\Var}{\mathrm{Var}}

\newcommand{\diam}{\mathrm{diam}}

\newcommand{\quadre}[1]{\left[ #1 \right]}
\newcommand{\tonde}[1]{\left( #1 \right)}

\newcommand{\di}{\mathrm d}

\def\be{\begin{eqnarray}}
	\def\ee{\end{eqnarray}}
\def\ben{\begin{eqnarray*}}
	\def\een{\end{eqnarray*}}

\newtheorem{theorem}{Theorem}[section] 
\newtheorem{lemma}[theorem]{Lemma}     
\newtheorem{proposition}[theorem]{Proposition}
\newtheorem{corollary}[theorem]{Corollary}

\theoremstyle{definition}
\newtheorem{definition}[theorem]{Definition}

\newtheorem{remark}[theorem]{Remark}

\pgfplotsset{compat=1.18}
\usepackage{subfig}

\usepackage{libertineRoman}
\usepackage{authblk}
\title{Fractal and Regular Geometry of Deep Neural Networks}

\author{Simmaco Di Lillo}
\author{Domenico Marinucci}
\author{Michele Salvi}
\author{Stefano Vigogna}
\affil{RoMaDS - Department of Mathematics, University of Rome Tor Vergata, 	Rome, Italy \newline
	{\tt{\{dilillo, marinucc, salvi, vigogna\}@mat.uniroma2.it}}}

\date{}                     

\renewcommand{\S}{\mathbb S}

\renewcommand{\P}{\mathbb P}

\newcommand{\cH}{\mathcal{H}}

\newcommand{\dH}{\dim_{\rm H}}
\newcommand{\W}{\Omega}
\newcommand{\w}{\omega}

\newcommand{\s}[1]{\left\{#1\right\}}
\newcommand{\nn}[1]{\| #1 \| ^2}
\newcommand{\n}[1]{\| #1 \| }

\DeclareMathOperator{\dvol}{dVol}

\DeclareMathOperator{\Li}{Li}
\DeclareMathOperator{\B}{B}

\usepackage{bbm}

\AddToHook{env/lemma/begin}{\crefalias{theorem}{lemma}}
\AddToHook{env/proposition/begin}{\crefalias{theorem}{proposition}}
\AddToHook{env/corollary/begin}{\crefalias{theorem}{corollary}}
\AddToHook{env/remark/begin}{\crefalias{theorem}{remark}}
\AddToHook{env/definition/begin}{\crefalias{theorem}{definition}}
\usepackage{titlesec}

\begin{document}
\maketitle

		\begin{abstract}
			We study the geometric properties of random neural networks by investigating the boundary volumes of their excursion sets for different activation functions, as the depth increases. More specifically, we show that, for activations which are not very regular (e.g., the Heaviside step function), the boundary volumes exhibit fractal behaviour, with their Hausdorff dimension  monotonically increasing with the depth. On the other hand, for activations which are more regular (e.g., ReLU, logistic and $\tanh$), as the depth increases, the expected boundary volumes can either converge to zero, remain constant or diverge exponentially, depending on a single spectral parameter which can be easily computed. Our theoretical results are confirmed in some numerical experiments based on Monte Carlo simulations.
		\end{abstract}
\noindent \textbf{MSCcodes}:  	60G60, 62B10, 62M45, 68T07.
\renewcommand{\contentsname}{Structure of the work}
\cleardoublepage
\tableofcontents
\clearpage

\titleformat{\section}[runin]
{\normalfont\large\bfseries}{\thesection}{1em}{}

	\section{Introduction}
	In recent years, neural networks have attracted considerable attention from mathematicians, leading to significant advances in their theoretical understanding.
	Efforts to analyse deep learning have drawn from various perspectives, including approximation theory \cite{DeVore_Hanin_Petrova_2021}, statistics \cite{MR4134774,Bartlett_Montanari_Rakhlin_2021}, and optimization~\cite{pmlr-v97-allen-zhu19a,NEURIPS2018_a1afc58c}.
	Besides these approaches, another line of research has centred on the study of geometric properties, with the goal of characterizing the complexity of neural networks across different architectures \cite{NIPS2014_109d2dd3,pmlr-v97-hanin19a,GOUJON2024115667,doi:10.1137/22M147517X,6697897,10.5555/3016100.3016187}.
	Meanwhile, building upon a classical result by Neal \cite{neal2012bayesian}, several papers have studied large-width networks at initialization through the lenses of Gaussian processes, proving central limit theorems at different levels of generality.

	From this latter point of view, one considers a neural network with suitably normalized random weights, and let the width grow to infinity. Neal first observed that any evaluation of a shallow network converges to a Gaussian random variable. Later on, this was extended to deep networks seen as random functions, showing convergence to a Gaussian process, see the functional central
	limit theorems \cite{g.2018gaussian, lee2018deep, 10.1214/23-AAP1933} and their quantitative counterparts \cite{cammarota_marinucci_salvi_vigogna_2023, basteri2024quantitative,favaro2025quantitative, celli2025entropic}, as well as large deviation properties \cite{MR2102887,andreis2025ldp,MR4994611}.
	As a consequence,
	random neural networks with large width are very close to Gaussian fields. In
	some sense, this could be even viewed as a characterization of neural networks for ``generic''
	values of their parameters, although it is not guaranteed that the networks will show a similar
	behaviour after training.

	Viewing neural networks as Gaussian processes opens up new ways of characterizing their geometric properties and, consequently, their complexity.
	In fact, a substantial body of literature, unrelated to machine learning, has been concerned with the study of geometric structures in Gaussian random fields, such as critical points and excursion sets. Classical references are the textbooks~\cite{book:73784,book:145181}, whereas some recent papers in this area include~\cite{AngstPoly2020,MR3857854,st00,st0,MR4538280,EstradeLeon2016,GassStecconi2024,AAP2021}.
	The aim of this paper is to apply these tools to neural networks.

	\smallskip

	A first step in this direction was taken in \cite{nostro}. The authors considered isotropic Gaussian random fields on the sphere $\S^d$ emerging from random neural networks, and studied their angular power spectrum as a function of the non-linear activation function $\sigma$ and the depth $L$ of the network.
	This explicit analysis led to a classification of networks into three classes: \begin{itemize}
		\item A low-disorder case, corresponding to fields that approach a constant  exponentially fast as $L$ grows to infinity. The convergence holds in $L^2$ and in all Sobolev spaces.
		\item A sparse case, where the fields approach a constant at polynomial speed. Here the convergence holds in $L^2$, but not in Sobolev norms.
		\item A high-disorder case, where the field does not converge at all.
	\end{itemize}
	A possible interpretation is that in the sparse case information propagates well  through the layers of a deep network during gradient descent, while low disorder networks lose information after a few layers and high disorder networks are too chaotic. The analysis allowed also to derive some useful insights on the importance of the ReLU activation function (which falls into the sparse regime), on the role of depth and on the effective dimension (namely, the degree needed for a polynomial to  approximate well the network) for a given architecture.

	\smallskip

	The present paper builds from that contribution to explore the geometric features of neural networks in the infinite-width limit.
	In particular, we investigate how the  geometry of their random excursion sets is influenced by the architecture (activation function and depth).
	We define the Covariance Regularity Index (CRI) of a network (\Cref{def::class}), describing the regularity of the kernel of its associated Gaussian field; we also explore the connection that exists between the CRI and the decay of the angular spectrum,  see \Cref{sufficient_condition}). The value of the CRI naturally leads to the classification of random neural networks into two classes with drastically different geometric behaviour:
	\begin{enumerate}
		\item[{\bf Fractal class:}] A CRI value smaller than $1$, corresponding to activation functions that are not regular  (e.g., the Heaviside activation), yields excursion volumes that show fractal behaviour, that is, non-integer Hausdorff dimension (\Cref{teo_fractal}). More than that, the Hausdorff dimension of the excursion sets grows as the depth $L$ increases, converging to the  dimension of the ambient space. In this sense, we can say that excursion sets of these random neural networks are ``more and more fractal'' for deeper and deeper architectures.
		\smallskip
		\item[{\bf Kac-Rice class:}] Networks with a CRI value larger than $1$ always exhibit excursion volumes of finite size (\Cref{th1}), with an explicit finite expectation. The behaviour of this expectation as the depth $L$ increases, though, is greatly influenced by the underlying activation function. Indeed, neural networks falling in this class are further split into three categories, corresponding to the three regimes of \cite{nostro} (\Cref{rem::class}): neural networks belonging to the low-disorder regime have excursion volumes converging to zero on average as $L$ grows to infinity; for neural networks in the high-disorder regime these expectations diverge exponentially; for the sparse regime (which includes the ReLU activation), the expected boundary volume of excursion sets is constant.
	\end{enumerate}

	All of our results are validated with numerical experiments based on Monte Carlo simulations.

	The proof strategies for the two classes are completely different and  rely on some results for Gaussian random fields that are new and of independent interest.

	\smallskip

	For the fractal class, we carefully investigate the strong local nondeterministic properties of the field, see \cite{st1,st4,st5,MR3298476} for an overview on this topic. We extend  to any dimension a number of results for non-regular Gaussian fields from \cite{marinucci2011random}, where the authors only dealt with fields having domain on the $2$-dimensional sphere. In particular, we compute the exact uniform modulus of continuity for an isotropic Gaussian random field with CRI smaller than $1$ (\Cref{teo7}); using this  modulus, we upper bound the Hausdorff dimension of the graph of the field by an explicit construction of a $\delta$-net covering, and obtain a lower bound using a standard argument based on the potential method. For the dimension of the boundary excursion set, the lower bound is computed using again the potential method, while for the upper bound we use a generalization to the sphere of a classical result on the dimension of sets intersections with parallel planes. The proof is then completed by observing that the CRI of Gaussian fields resulting from neural networks in the fractal class decreases exponentially fast to 0 as L goes to infinity.

	\smallskip

	For the regular class, the idea is to use the classical Kac-Rice approach (see \cite{book:73784, book:145181}). In particular, we use the results from the recent paper~\cite{armentano2023generalkacriceformulameasure} and combine them together with the characterization of the angular power spectrum derived in~\cite{bietti2021deep}. In order to apply the Kac-Rice machinery, though, one needs to check that the Gaussian field is $C^1$-smooth; this property is well known to be satisfied when the field has CRI equal to $2$ (that is, when its kernel is $C^2$ in the whole interval $[-1,1]$), but this leaves out fields induced by activation functions (including ReLU) yielding a CRI between $1$ and $2$. Thus we fill this gap by proving that any Gaussian field with CRI larger than $1$ is almost surely $C^1$, see \Cref{th0}. This result is obtained through a stereographic projection of the field and a thorough analysis of its mean-square derivatives, leading to an application of Kolmogorov’s continuity theorem.

	Once we have established that the formulas for calculating the expected excursion volumes from~ \cite{armentano2023generalkacriceformulameasure} can be applied to any neural network in the Kac-Rice class, we explicitly calculate these expectations, proving that their value only depends on the derivative of the kernel in a neighbourhood of $1$. This value, in turn, is exactly what distinguished the three regimes appearing in~\cite{nostro}.

	\smallskip

	The plan of the paper is as follows: in~\Cref{back} we provide background material and the definitions that we use throughout the paper. ~\Cref{main} contains all our main results, including fractal behaviour and finite volumes of the excursion boundary.  We show the main steps of the proof for the fractal case in~\Cref{fractal_proof}, and for the finite volumes case in~\Cref{kac-rice_proof}.  More technical results and auxiliary lemmas are collected in the Appendix. In~\Cref{numerical} we illustrate our results by numerical simulations in dimension $2$.

	\section{Background and definitions}\label{back}

	In this section, we recall some standard background material and we introduce the notation used throughout the article.

	\subsection{Spherical random fields}\label{smot}

	Let $(\Omega, \mathcal F, \P)$ be a fixed probability space and let $T : \S^d \times \Omega \to \R $ be a finite-variance, zero-mean, isotropic Gaussian random field on the $d$-dimensional unit sphere $\S^d$.
	This means in particular that the following property holds:
	for every $g\in \mathrm{SO}(d+1)$ (the special group of rotations on $\R^{d+1}$), for every $ k \geq 1 $ and every $x_1, \dotsc, x_k \in \S^d$,
	$$(T(x_1), \dotsc, T(x_k) )  \sim (T(gx_1), \dotsc, T(g x_k) ) $$
	where $\sim $ denotes equality in distribution.

	By isotropy, the covariance function can be expressed as a function of the angle between two points, i.e. there exists $\kappa : [-1, 1] \to \mathbb{R}$ such that
	$$ \E[T(x) T(y)]  = \kappa(\langle x , y \rangle) \ , $$
	where $\langle \cdot, \cdot \rangle$ denotes the scalar product in   $\R^{d+1}$.

	Hence, by Schoenberg's theorem~\cite{schoenberg}, the covariance can be expressed as
	\begin{equation}\label{cova}
		\kappa(\langle x, y \rangle ) = \sum_{\ell=0}^\infty C_\ell \frac{n_{\ell,d}}{\omega_d} G_{\ell;d}(\langle x, y \rangle)   \ ,
	\end{equation}
	where
	\begin{itemize}
		\item $n_{\ell,d}$ is the number of linearly independent homogeneous harmonic polynomials of degree $\ell$ in $d+1$ variables given by $n_{0,d} =1$ and for $\ell>0$
		\begin{align*}
			n_{\ell,d} &= \frac{2\ell + d-1 }{\ell} \binom{\ell+d-2}{\ell-1} \ ;
		\end{align*}
		\item  $(C_\ell)_{\ell \in \N}$ is a sequence of non-negative numbers called the \emph{angular power spectrum} of $T$;
		\item $\w_d$ is the surface volume  of $\S^d$ given by
		\begin{equation}\label{surface}	\w_d = \frac{2\pi^{\frac{d+1}{2}}}{\Gamma\left(\frac{d+1}{2}\right)} \ ;
		\end{equation}
		\item  $(G_{\ell;d})_{\ell \in \N}$ is a sequence of polynomials known as normalized Gegenbauer polynomials (see \Cref{gege} for more details).
	\end{itemize}
	The following spectral decompositions
	hold in $L^2( \Omega\times \S^d)$ (see \cite{yadrenko1983spectral}, Theorem 1):
	$$ T(x,\w) = \sum_{\ell=0}^\infty \sum_{m =1}^{n_{\ell,d}} a_{\ell m }(\w) Y_{\ell m } (x) \ , \quad x\in \S^d \ , \ \w \in \Omega \ , $$
	where
	\begin{itemize}
		\item 		$(a_{\ell m})_{\ell\in \N, \ m = 1, \dotsc, n_{\ell,d}}$ is a triangular sequence of real-valued random variables such that
		$$ \E[a_{\ell m} a_{\ell' m'} ] = C_\ell \delta_{\ell,\ell'} \delta_{m,m'}   \ ; $$
		\item $(Y_{\ell m})_{\ell\in \N, \ m = 1, \dotsc, n_{\ell,d}}$ is an orthonormal sequence of real spherical harmonics, i.e.
		$Y_{\ell m}:\, \S^d \to \R $ are eigenfunctions of the Laplace-Beltrami operator and
		$$ \int_{\S^d} Y_{\ell m} (x) Y_{\ell' m'}(x) \dvol(x) = \delta_{\ell,\ell'} \delta_{m,m'} \ , $$
		where $\dvol$ denotes the $d$-dimensional volume form on $\S^d$. See \Cref{a2} for an explicit construction. If $N_d = (0, \dotsc, 0,1)\in \S^d$ denotes the north pole, one can choose the spherical harmonics such that
		$$ Y_{\ell,m}(N_d )	 = \sqrt{\frac{n_{\ell,d}}{\w_d}} \delta_{m,1} \ .
		$$
		See~\Cref{polo-nord} for more details.

	\end{itemize}

	\subsection{Random neural networks}
	For $L\geq1$, let $ n_0=d , n_1 , \dots , n_L , n_{L+1} $ be positive integers. Let $ \sigma : \R \to \R $ be a (non-linear) function.
	A neural network of depth $L$, widths $ n_1 , \dots  ,n_L $ and activation $\sigma$ is a function $T_L:\, \S^d \to \R^{n_{L+1}}$ defined by the following recursive formula:
	$$ T_s(x) =\begin{cases}W^{(0)}x + b^{(1)}, & \text{ if } s=0\\
		W^{(s)}\sigma(T_{s-1}(x)) + b^{(s+1)}& \text{ if } s =1, \dots, L \ ,
	\end{cases}
	$$
	where $W^{(s)}\in \R^{n_{s+1}\times n_s}$ and $b^{(s)} \in \R^{n_s}$ are called weights and biases, respectively,
	and $\sigma$ is understood to be applied entry-wise.
	We say that $T_L$ is a \emph{random} neural network if weights and biases are independent random variables with independent components. For $\Gamma_b\in [0,1)$ we take the following standard calibration condition:
	$$ W^{(0)}_i \sim \mathcal N(0,1-\Gamma_{b}), \quad W^{(s)}_{i,j} \sim \mathcal N(0, \Gamma_W n_s^{-1/2}), \quad  b^{(s)}_i \sim \mathcal N(0, \Gamma_b) \ , $$
	where
	$$\Gamma_W =\frac{1-\Gamma_b}{ \E[\sigma(Z)^2]} \qquad Z\sim \mathcal N(0,1)\ . $$
	It is well-known (see~\cite{pmlr-v97-hanin19a, 10.1214/23-AAP1933, basteri2024quantitative} and references therein) that
	the random neural network $T_L$ converges weakly to an isotropic, zero-mean,  Gaussian random field with $n_{L+1}$ i.i.d. components
	as the widths $n_1,\dots,n_L$ go to infinity.
	In particular, if $K_L$ is the covariance function of one component of the limiting process, we have
	\begin{equation} \label{eq:KL}
		\begin{aligned}
			& K_L(x,y) = \kappa_L(\langle x, y \rangle) = \underbrace{\kappa\circ\cdots \circ \kappa}_{L \text{ times}}(\langle x, y \rangle) \\
			& \kappa(u) = \E\quadre{ \sigma(Z_1)\sigma(uZ_1 + \sqrt{1-u^2} Z_2)} \ ,  \qquad (Z_1,Z_2) \sim \mathcal N\tonde{0,I_2} \ .
		\end{aligned}
	\end{equation}
	We note that the previous calibration condition ensures that $\kappa_L(1)=1$ for all $L$. \\

	The infinite-width limit of neural networks is representative of overparametrized regimes which are typical in practice.
	In view of this, in the following we will deal with the class of Gaussian random fields on the sphere with covariance of the form \eqref{eq:KL}.
	With a slight abuse of terminology, the limit Gaussian fields will also be called random neural networks.
	Moreover, for the sake of simplicity and without loss of generality, we will set $ n_{L+1} = 1 $.

	\section{Main results}\label{main}

	\subsection{Spectral index and activation functions}
	Our first definition will be used to characterize the regularity of the covariance function.
	\begin{definition}[Covariance Regularity Index]\label{CRI}
		Let $T$ be a random field and let $\kappa$ be its  covariance function.
		Assume that $\kappa\in C^2((-1,1))$.
		The CRI of $T$ is the largest $ \beta \in (0,2] $ such that
		\begin{itemize}
		\item  $\kappa$ admits the following expansions around $\pm 1$ as $ t \to 0^+ $:
					\begin{equation}\label{dec::bb}\begin{aligned} & \kappa(1-t) = p_1(t) + c_1 t^{\beta} + o (t^{\beta})\ , \\
					& \kappa(-1+t) = p_{-1}(t) + c_{-1} t^{\beta} + o (t^{\beta}) \ ,
				\end{aligned}
			\end{equation}
			where $p_1, p_{-1}$ are polynomials of degree strictly smaller than $\beta$, and $ c_1,c_{-1} \ne 0 $ are constants;
		\item $\kappa'$ and $\kappa''$ admit similar expansions obtained by differentiating \eqref{dec::bb}.
				\end{itemize}
	\end{definition}

	Note that, if $ \kappa $ is $\Delta$-analytic, then the last condition on the derivatives of $\kappa$ can be omitted.
We also remark that the CRI is equal to $2$ if and only if $ \kappa \in C^2([-1,1]) $.
		More in general, a CRI equal to $\beta$ implies that $\kappa$ is H\"older continuous of order $\beta$.
	It follows that a higher CRI of the random field implies higher H\"older regularity of the covariance function,
	and thus of the field itself (see \Cref{th0} and \cite[Proposition 3.3]{dilillo2025criticalpointsrandomneural} for a generalization).\\

	We want to establish a connection between the decay of the angular power spectrum of $T$ and its CRI.
	To this end, we introduce the concept of \emph{spectral index}.
	\begin{definition}[Spectral Index]\label{spectral}
		Let $(C_\ell)_{\ell \in \N}$ be the angular power spectrum of a random field $T$ such that
		\begin{equation*}\label{Cell} C_\ell = q(\ell^{-1}) \ell^{-(\alpha+d)} \ ,\qquad \ell=1, 2, \dotsc,
		\end{equation*}
		where $q$ is a  polynomial such that $q(z)>0$ for all $z\in \R$.
		Then we say that $T$ has spectral index $\alpha$.
	\end{definition}

	\begin{remark} Since the field $T$ has finite variance and $n_{\ell,d}\sim \ell^{d-1}$ as $\ell \to + \infty$, by~\eqref{cova} we have $\alpha>0$.
	\end{remark}

	The following proposition links the two definitions and it is proved in \Cref{ciccio}.
	\begin{proposition}\label{sufficient_condition} Let $\alpha\neq 2$ be the spectral index of  an isotropic Gaussian random field  $T$. Then the CRI of $T$ is equal to $\min(2, \alpha/2)$.
	\end{proposition}

	\begin{remark}Using the computation in~\Cref{ciccio}, for $\alpha=2$ we have
		$$ \kappa(1-t) = 1 + B t\ln(t) + o(t\ln(t)) \ , $$
         that is, there is a logarithmic correction.
	Hence \Cref{sufficient_condition} does not hold for $\alpha = 2 $.
	\end{remark}

	The following definition allows a  comprehensive classification on the geometric features induced by  different networks architectures.

	\begin{definition}[Fractal and Kac-Rice class]\label{def::class}
		A function $\sigma:\, \R \to \R $  belongs to
		\begin{itemize}
			\item
			the \emph{fractal class}  if the corresponding random neural networks have CRI in $(0,1)$.
			\item the  \emph{Kac-Rice class} if the corresponding random neural networks have CRI in $(1,2]$.
		\end{itemize}
	\end{definition}

\begin{remark} The critical case CRI equal to $1$ is not included in any of the two classes above
because, as we will show in \Cref{teo_fractal}, the excursion sets are not fractal, nor the field is sufficiently regular to use the Kac-Rice formula.
\end{remark}

	\begin{remark}[Some examples] Using the computations in \cite{bietti2021deep},  eq. (12) and (13), it follows that the Heaviside activation function $\sigma(x) = \mathbb I_{[0,+\infty]}(x) $   leads to the fractal class, whereas ReLU and leaky-ReLU belong to the Kac-Rice class. Moreover, it is easy to prove that a $C^2$ activation function (e.g. logistics, hyperbolic tangent, etc.) belongs to the Kac-Rice class.
	\end{remark}
	\begin{remark} It is important to stress that, after training, the geometry of the random neural networks can be different. It can be shown that in the ``lazy training regime'' (i.e. for the neural tangent kernel) the ReLU activation induces excursion sets that exhibit fractal geometry. These results are currently being investigated and will be the object of future work (see \cite{dilillophd}).
	\end{remark}

	\begin{remark}It follows from the results in the next section that the fact that a given $\sigma$ belongs to one of the two classes does not depend on the depth of the network or on the value of $\Gamma_b$. On the other hand, the depth has an important influence on the Hausdorff dimension of the boundary of the excursion sets in the fractal class and on their volumes in the Kac-Rice class.
	\end{remark}

	\subsection{The fractal class}
	Let
	$$ \Gamma_T= \s{ (x, T(x))\; | \;  x\in \S^d}$$
	be the graph of $T$, and  let
	$$ T^{-1}(u)= \{ x \in \S^d \; |\; T(x) = u\} $$
	be the hypersurface of $T$ at level $u$. \\

	Our first result covers the case where the activation function is not regular enough to ensure that the excursion sets have integer Hausdorff dimension, i.e. the fractal regime.
	Throughout the paper, we will denote by $\dH$ the Hausdorff dimension (see~\Cref{Haus} for more details).
	\begin{theorem}\label{teo_fractal} Let $\sigma $ be an activation function in the fractal class. For any $L$, let  $T_L$ be a random neural network with activation $\sigma$ and depth $L$.  If $\beta \in(0,1]$ is the CRI of $T_1$,  then  the CRI of $T_L$ is $\beta^L$.  In particular, the CRI  is in $(0,1]$ for any $L$. Moreover,
		\begin{itemize}\item with probability one,
			\begin{equation} \label{grafoqqq} \dH (\Gamma_{T_L})  = d+1 -\beta^L \ ;
			\end{equation}
			\item for every $u\in \R$,
			\begin{equation}\label{almost} \dH(T_L^{-1}(u))= d -\beta^L
			\end{equation}
with positive probability.
		\end{itemize}

	\end{theorem}

	\begin{remark}
\Cref{teo_fractal}	 implies that  for $\beta\in(0,1)$ and for all $ u \in \R $,
		$$
		\E\big[\cH^{d-1}(T_L^{-1}(u)) \big]  =+\infty   \ . $$

	\end{remark}

	\begin{remark}Note that, if $\beta\in (0,1)$, then  $\beta^ L \to 0$ as $L\to + \infty$, and the Hausdorff dimension of $T_L^{-1}(u)$ converges to the Hausdorff dimension of the ambient space.
	\end{remark}

	\subsection{The Kac-Rice class}
	In this subsection, we study the case where the random neural network is regular enough to exhibit standard, integer-valued  Hausdorff dimension of the boundary volumes (namely $\dH(T_L^{-1}(u)) = d-1$). On the other hand, the expected values of these volumes depend heavily on the depth of the network, as illustrated by~\Cref{th1}.

	\begin{theorem}\label{th1}Let $\sigma$ be an activation function in the Kac-Rice class.  For any $L$, let  $T_L$ be a random neural network with activation $\sigma$ and depth $L$. Then
		$$\E\quadre{ \mathcal H^{d-1} (T_L^{-1}(u))}  = \w_{d-1} \kappa'(1)^{L/2} e^{-u^2/2}  $$
		where $\w_{d-1}$ is the surface volume of $\S^{d-1}$ (cfr.~\eqref{surface}).
	\end{theorem}
	\begin{remark}\label{rem::class}
	In the recent paper \cite{nostro}, neural networks were divided into three classes with different $L^2$ convergence regimes: low-disorder ($\kappa'(1)<1$), sparse ($\kappa'(1)=1$) and high-disorder ($\kappa'(1)>1$).
	\Cref{th1} confirms geometrically what \cite{nostro} established analytically.
	In particular, in the low-disorder case, when the field converges (in $L^2$) to a constant as $ L \to \infty $, the expected boundary volumes converge to $0$.
	In the sparse case, when the field converges to a constant but its derivatives do not vanish, the boundary volumes remain constant on average.
	In the high-disorder case, when the field does not converge to a constant and its derivatives diverge exponentially, the expected boundary volumes also diverge exponentially fast.
	\end{remark}

	To prove the previous results we need that  $ T_L \in C^1(\S^d) $ a.s.; this  is well-known to be satisfied when $\kappa\in C^2([-1,1])$ (i.e.  $T_L$ has CRI equal to $2$) (see e.g. \cite{lang-schwab}). However, this leaves out important cases, in particular the celebrated ReLU activation function.  Our next result shows that, if $\sigma$ belongs to the Kac-Rice class, then $T_L\in C^1(\S^d)$ a.s..  This result has independent interest for the theory of the geometry of random fields, because it allows us to exploit the Kac-Rice approach in more general settings than the ones given for instance in~\cite{book:73784,book:145181}.

	\begin{proposition}\label{th0}
		Let $ T : \S^d \to \R $ be an isotropic Gaussian random field with zero mean and unit variance. If $T$ has CRI greater than $1$, then
		$ T \in C^1(\S^d) $ almost surely.
	\end{proposition}
	\section{Proof of \Cref{teo_fractal}: fractal class}\label{fractal_proof}
	Before we proceed with the proof, we recall some more background material.
	\subsection{Hausdorff dimension and measure}\label{Haus}
	We briefly recall the definition of Hausdorff dimension and some of its well-known properties (see~\cite{falconer2013fractal} for more details). \\

	Let $(X,d)$ be a metric space. For $U\subseteq X $ a non-empty set, its diameter is defined as
	$$ \diam(U) = \sup_{x,y\in U}  d(x,y) \ .$$
	We say that a countable (or finite) collection of sets  $(U_i)_{i\in I}$   is a $\delta$-cover of $Y\subseteq X$ if,  for all $i\in I$, $\diam(U_i)\leq \delta $  and $Y\subseteq \bigcup_{i\in I} U_i $.

	Let $Y\subseteq X$, and let $s$ be a non-negative number. For any $\delta>0$, we define
	$$ \cH^s_\delta(Y) = \inf \s{ \left. \sum_{i\in I} \diam(U_i)^s  \right|
		(U_i)_{i\in I} \text{ $\delta$-cover of $Y$}}$$
	and the $s$-dimensional Hausdorff measure of $Y$ as
	$$ \cH^s (Y) = \lim_{\delta\to 0} \cH^s_\delta(Y)  \ . $$
	We define the Hausdorff measure of $Y$ as
	$$ \dH (Y) = \sup_{s\geq 0}\s{ \cH^s(Y) = \infty} \ , $$
	or, in an equivalent way,
	$$ \dH (Y) = \inf_{s\geq 0}\s{ \cH^s(Y) = 0} \ .$$

	The following theorem, whose proof can be found in~\cite{falconer2013fractal}, Theorem 4.13, allows us to obtain a lower bound for the Hausdorff dimension of a set.
	\begin{theorem}\label{potential} Let $F$ be a subset of $\R^n$. If there exists a mass distribution $\mu$ on $F$ such that
		$$ \int_F \int_F \frac{1}{|x-y|^s} \di \mu(x) \di \mu(y) <\infty$$
		for some $s\geq 0$, then $\dH(F)\geq s$.
	\end{theorem}
\begin{remark}\label{salvi}	The proof of the previous theorem is given only for subsets of $\R^n$.  In the following we use this theorem also for $F\subseteq \S^d \times \R$, using the fact that  the geodesic distance  on $\S^d$ and the restriction on $\S^d$ of the Euclidian metric are equivalent.  Furthermore, since for two metric spaces   $(X,d_X)$ and $(Y,d_Y)$ the  metrics
\begin{equation*}
	\begin{aligned} &d_2((x,y), (x',y')) = \sqrt{ d_X(x,x')^2 + d_Y(y,y')^2} \, , \\
&d_\infty ((x,y), (x',y')) = d_X(x,x')\vee d_Y(y,y')
\end{aligned}\qquad
x,x' \in  X,\; y,y'\in Y
\end{equation*}
	are equivalent, we will from now on use $d_\infty$ on product spaces unless specified otherwise.
\end{remark}
	The following proposition generalizes to Borel subsets of $\S^d \times \R$ the result proved for Borel subsets of $\R^2$ in~\cite{falconer2013fractal}, Proposition 7.9.  The proof follows the one presented in~\cite{falconer2013fractal}, with slight modifications. Once this result is established, we obtain \Cref{cor_updim} in the same way as in~\cite{falconer2013fractal}, Corollary 7.10.

	\begin{proposition}Let $F$ be a Borel subset of  $ \S^d\times \R$. Then, for every $1\leq s \leq d+2$ we have
		$$ \cH^s(F)\geq  \int_\R  \cH^{s-1}(F\cap L_x) \di x $$
		where $L_x = \s{(t,x) \; | \; t\in \S^d}\subseteq \S^d \times \R$.
	\end{proposition}
	\begin{proof}Following~\Cref{salvi}, we  consider $\S^d\times \R$ seen as a subset of $\R^{d+2}$ with the infinity metric $d_\infty$.  Let $\delta>0$. For any $\varepsilon>0$, let $(U_i)_{ i \in I }$ be a  $\delta$-cover of $F$ such that
		$$ \sum_{i\in I} \diam(U_i)^s \leq \cH^s_{\delta}(F) + \varepsilon \ . $$
For every $j=1, \dots, d+2$, let $\pi_j:\, \R^{d+2} \to \R$ be the projection on the $j$-th component. Since we are using $d_\infty$, for any $i \in I $ and $j =1, \dots, d+2$, $$\diam(\pi_j(U_i)) \leq   \diam(U_i)\, .$$
It follows that, for every $i\in I$ and  $j=1, \dots, d+2$, there exists $Q_j^i\subseteq \R $ such that $\pi_j(U_i)\subseteq Q_j^i $ and $\diam(Q_j^i) = \diam(U_i)$. Then
$$ U_i \subseteq \prod_{j=1}^{d+2}Q_j^i =S_i \ . $$
We note that $\diam(S_i) = \diam(U_i)$ and hence  $(S_i \cap L_x)_{i\in I}$ is a $\delta$-cover of $F\cap L_x $. Let
$$ I_x = \{ i \in I \;  | \; x\in Q_{d+2}^i  \} \, .$$
Then
$$ S_i \cap L_x = \begin{cases} \prod_{j=1}^{d+1} Q_j^i \times \s{  x} & \text{ if } i \in I_x\\
	\emptyset & \text{ otherwise}
\end{cases}
$$
and in particular
$$ \diam(S_i \cap L_x) =\begin{cases} \diam(U_i) & \text{ if } i \in I_x\\
	0 & \text{ otherwise}
\end{cases}\, . $$
When $I_x\neq \emptyset$, one has
			\begin{align*} \cH^{s-1}_\delta(F \cap L_x) \leq &\sum_{i\in I}\diam(S_i \cap L_x)^{s-1}  = \sum_{i\in I_x}\diam(S_i \cap L_x)^{s-1} \\
			= & \sum_{i\in I_x} \diam(S_i \cap L_x)^{s-d-2} \diam(S_i\cap L_x)^{d+1}\\
			=  &  \sum_{i\in I_x} \diam(U_i)^{s-d-2}  \int_{\R^{d+1}} \mathbb I_{S_i\cap L_x}((t,x)) \di t \\
			=&  \sum_{i\in I} \diam(U_i)^{s-d-2}  \int_{\R^{d+1}} \mathbb I_{S_i\cap L_x}((t,x)) \di t \ .
		\end{align*}
We note that the previous inequality holds also when $I_x=\emptyset$, since both left-and right-hand sides are zero.
		Thus,
		\begin{align*} \int_\R \cH^{s-1}_\delta (F\cap L_x) \di x &\leq \sum_{i\in I} \diam(U_i)^{s-d-2} \int_\R\int_{\R^{d+1}} \mathbb I_{S_i\cap  L_x }((t,x))  \di t \di x \\
			& =\sum_{i\in I} \diam(U_i)^{s-d-2}\diam(U_i)^{d+2} \leq\mathcal H^{s}_\delta (F)+\varepsilon \ .
		\end{align*}
		Since $\varepsilon$ is  arbitrary,  we have
		$$ \int_\R \cH^{s-1}(F\cap L_x) \di \leq \cH^{s}_\delta(F)$$
		and the claim follows by letting $\delta \to 0 $.

	\end{proof}
	\begin{corollary}\label{cor_updim}
		Let $F$ be a Borel subset of $\S^d\times \R$. Then, for almost all $x$ (in the sense of $1$-dimensional Lebesgue measure),
		$$\dH (F\cap L_x) \leq \max(0, \dH(F) -1)\ .$$
	\end{corollary}

	\subsection{Technical tools} In this section we give two technical tools which are useful to compute the Hausdorff dimension.  The first result
	generalizes  \cite{marinucci-xiao}, Theorem 2 from dimension $2$ to any dimension. Its proof is rather technical and follows the argument in~\cite{marinucci-xiao} (see~\Cref{a2} for more details). The second tool that we will use in the proof of~\Cref{teo_fractal} is the construction of a $\delta$-net on the sphere; in our second technical lemma, we prove a bound for its cardinality.
	\begin{proposition}
		\label{teo7}
		Let $T$ be an isotropic Gaussian random field with CRI $\beta \in (0,1]$. Then there exists a constant $K> 0$ such that, with probability one,
		$$
		\lim_{\varepsilon \to 0} \sup_{x,y \in \S^d \atop{d_{\S^d}(x,y) \leq \varepsilon}} \frac{|T(x) - T(y)|}{d_{\S^d}(x,y)^{\beta} \sqrt{
				|\ln d_{\S^d}(x,y)|}} \leq  K  \ ,
		$$
		where $d_{\S^d}$ denotes the geodesic distance on the sphere, i.e. $ d_{\S^d}(x,y) = \arccos(\langle x,y \rangle) $.
If $ \beta < 1 $, then the display above holds as an equality.
	\end{proposition}
	Let $S$ be a subset of $\S^d$, and let $\delta>0$. We say that $S$ is $\delta$-separated if $d_{\S^d}(x,y)> \delta$ for all $ x,y \in S $ with $ x \ne y $; furthermore, we say that $S$ is maximal if $S\cup \s{x} $ is not $\delta$-separated anymore, for any $x\in \S^d$.
	\begin{lemma}\label{separ}For all $\delta>0$, there exists a  maximal $\delta$-separated set $S_\delta $ of $\S^d$.
		Moreover, for $\delta \to 0 $ we have
		$$ C_d \delta^{-d} + O(\delta^{2-d}) \leq |S_\delta|
		\leq  2^d C_d  \delta^{-d} +O(\delta^{2-d}) \ , $$
		where
		$C_d = 	\frac{d}{2} \B(d/2,1/2) $ with $\B$ denoting the Beta function,
		and $|S_\delta|$ is the cardinality of~$S_\delta$.
	\end{lemma}
	\begin{proof}We split the proof into three parts.
		\begin{itemize}
			\item
			Let $\mathcal S$ be the set of all $\delta$-separated sets. Let $x\in \S^d$; then $\s{ x} $ is $\delta$-separated. Hence $(\mathcal S, \subseteq )$ is a non-empty poset. If we prove that every  chain in $\mathcal S$ has an upper bound in $\mathcal S$, the first part of the claim follows by Zorn's lemma.
			Let $\mathcal{C} \subseteq \mathcal S$ be a chain and $U = \bigcup \mathcal C$; then $U\in\mathcal S$. Indeed, for all $x, y\in U$ there exist $C, D\in \mathcal C$ such that $x\in C$ and $y\in D$. Since $\mathcal C$ is a chain, we can assume without loss of generality that $C \subseteq D$. Hence $x,y\in D$. Since $D\in \mathcal S$, $D$ is $\delta$-separated, and we have $d_{\S^d}(x,y)> \delta$.
			\item
Let $S=S_\delta$ be a maximal $\delta$-separated set.
			We now prove the lower bound for the cardinality of $S$. Let $B_\delta (x) = \s{ y\in \S^d \; | \; d_{\S^d}(x,y)\leq \delta}$.
			Then, by the maximality of $S$,
			$$ \S^d= \bigcup_{s\in S} B_\delta(s) $$
			and hence
			\begin{equation}\label{stima1}
				\cH^{d}(\S^d ) = \cH^d\tonde{\bigcup_{s\in S } B_\delta(s)} \leq |S| \cH^d(B_\delta(x)) \ ,
			\end{equation}
			where $x$ is an arbitrary point on $\S^d$.
			Now, using the spherical coordinates we have
			$$ \cH^d(\S^d)= \int_{[0,\pi]^{d-1} \times [0,2\pi]} \prod_{i=1}^{d-1} \sin(\theta_i)^{d-i}  \di \theta_1 \dotsb \di \theta_d \ , $$
			and
			$$ \cH^d(B_\delta(x)) = \int_0^\delta \sin(\theta_1)^{d-1} \di \theta_1 \int_{[0,\pi]^{d-2}\times [0,2\pi]} \prod_{i=2}^{d-1} \sin(\theta_i)^{d-i} \di \theta_2\dotsb \di \theta_d  \ .
			$$
			Using~\eqref{stima1}  we obtain
			$$ |S| \geq \frac{\int_0^\pi \sin(\theta)^{d-1} \di \theta}{\int_0^\delta \sin(\theta)^{d-1} \di \theta} \ .$$
			For the numerator, using the change of variable $\cos(\theta)^2= t$ we obtain
			\begin{align*} \int_0^\pi \sin(\theta)^{d-1} \di \theta = &\frac{1}{2}\int_0^1 (1-t)^{\frac{d-2}{2}} t^{-1/2} \di t \\
				= &\frac{1}{2} \B(d/2,1/2) \ .
			\end{align*}
			For the denominator, we note that
			for $\delta\to 0 $ we have $\sin(\theta) = \theta + O(\theta^3)$, hence
			$$ \int_0^\delta \sin(\theta)^{d-1} \di \theta = \frac{\delta^{d}}{d}+O(\delta^{d+2}) \ . $$
			\item Let us prove the upper bound.  We note that
			$$\bigcup_{s\in S} B_{\delta/2}(s) \subseteq \S^d \ , $$ hence
			$$ \cH^d(\S^d)\geq \cH^d\tonde{\bigcup_{s\in S } B_{\delta/2}(s)} = \sum_{s\in S } \cH^d(B_{\delta/2}(s)) = |S| \cH^d(B_{\delta/2}(x)) $$
			for an arbitrary $ x \in \S^d$,
			where for the first equality we used the fact that the balls are disjoint.
			Indeed, since $S$ is $\delta$-separated,  for all $s, t\in S$ we have $d_{\S^d}(s,t)>\delta$ and so $B_{\delta/2}(s)\cap  B_{\delta/2}(t) = \varnothing$. Using the same computation as in the previous item we obtain the claim. \qedhere	\end{itemize}
	\end{proof}
	\subsection{The CRI of deep neural networks}In this section, we prove that the fractal class is stable under the depth $L$ and we show how the Hausdorff dimension increases with $L$.

	Let $\sigma$ be in the fractal class and let $\beta \in (0,1)$ be the CRI  of $T_1$.
	We will prove by induction on $L\geq 1$ that
	$$ \kappa_L(1-t) = 1 - c_1^{\frac{ 1-\beta^L}{1-\beta}} t^{\beta^L} +o\big(t^{\beta^L}\big)\ . $$
	For $L=1$, the claim coincides with \eqref{dec::bb}, since in the case $\beta\in (0,1)$ one has $p_1\equiv 1$, bacause $\kappa(1) =1$. For $L\geq 1$ we have
	\begin{align*} \kappa_{L+1}(1-t) = &\kappa(\kappa_L(1-t))
		= \kappa \Big( 1 -c_1^{\frac{ 1-\beta^L}{1-\beta}} t^{\beta^L} + o\big(t^{\beta^L}\big)\Big) \\
		=&1 -c_1^{(1+\beta) \cdot  \frac{ 1-\beta^L}{1-\beta}} t^{\beta^{L+1}} + o\big(t^{\beta^{L+1}}\big) \\
		=&1 -c_1^{ \frac{ 1-\beta^{L+1}}{1-\beta}} t^{\beta^{L+1}} + o\big(t^{\beta^{L+1}}\big)
		\  ,
	\end{align*}
	where in the second equality we used the inductive hypothesis, and in the third equality we used \eqref{dec::bb}.
	A similar computation shows that
	$$ \kappa_L(-1+t) = 1 - c_{-1}^{\frac{ 1-\beta^L}{1-\beta}} t^{\beta^L} +o\big(t^{\beta^L}\big)\ . $$
	Therefore, the CRI of $T_L$ is equal to $\beta^L$.
	\subsection{Proof of \eqref{grafoqqq}: Hausdorff dimension of the graph}
	To compute the Hausdorff dimension of $\Gamma_T$, we prove an upper bound by exhibiting an explicit cover,
	and a matching lower bound using a standard argument based on the potential method (see~\Cref{potential}).

	\smallskip

	In view of~\Cref{teo7},  for any $\delta>0$ small enough and any $x,y\in \S^d$ with  $d_{\S^d}(x,y) \leq \delta$,
	\begin{equation}\label{Hold} |T(x) - T(y)| \le C \delta^{\beta}
		\sqrt{|\log(\delta)|}
	\end{equation}
	holds almost surely for some constant $C>0$. Let $S_\delta$ be a maximal $\delta$-separated subset of $\S^d$, and let $N$ be the smallest integer such that
	$$ N \frac{\delta} 2 \geq C \delta^{\beta} \sqrt{|\log(\delta)|} \ .$$
	Let
	\begin{equation*}
		\mathcal B_\delta = \s{ B_\delta\Bigg(\bigg(z,T(z)+k \frac{\delta}2\bigg)\Bigg) \; \Bigg| \; z\in S_\delta, \ k\in \Z , \, |k|\leq N }  \   ,
	\end{equation*}
	where $ B_\delta((x,s)) $ denotes the ball of centre $(x,s) \in \S^d\times\R$ and radius $\delta$ with respect to the metric
	$$ d((x,t), (y,s)) =  d_{\S^d}(x,y)\vee |t-s|  \qquad x,y\in \S^d, \; t,s\in \R \ . $$
	Then $\mathcal B_\delta$ is a $\delta$-cover of  $\Gamma_T$.  Indeed, for any $(x, T(x)) \in \Gamma_T$, since $S_\delta$ is a maximal $\delta$-separated set, there exists $z\in S_\delta$ such that $d_{\S^d}(x,z)\leq \delta$. By~\eqref{Hold}, we have
	$$ | T(x) - T(z) | \leq C \delta^{\beta} \sqrt{|\log(\delta)|}\leq N\frac \delta 2   \ . $$
	Let $s_\varepsilon= d+1-\beta+\varepsilon $ for some $\varepsilon>0$.
	Then, using \Cref{separ} and the definition of $N$ we have
	\begin{align*}\cH_\delta^{s_\varepsilon} (\Gamma_T)&\leq  \sum_{B \in \mathcal{B}_\delta} \diam(B)^{s_\varepsilon} = |\mathcal B_\delta| (2\delta)^{s_\varepsilon}   = |S_\delta| (2N+1)(2\delta)^{s_\varepsilon} \\
		&\leq  2^{d+2+s_\varepsilon} C_d \delta^{s_\varepsilon-d} (C \delta^{\beta-1} \sqrt{\log|\delta|} + 1) (1 + O(\delta^2))\\
		&\leq C' \delta^{ \varepsilon} \sqrt{|\log(\delta)|} + O\tonde{\delta^{1+\varepsilon-\beta}} \ ,
	\end{align*}
	where $C'$ is a constant depending on $ d , \varepsilon, \beta $.
	Since $0<\beta<1$, we have $ \cH^{s_\varepsilon}(\Gamma_T) =0 $. Since $\varepsilon$ is arbitrary, we get the upper bound
	$$ \dH(\Gamma_T)\leq d+1 -\beta \ . $$
	To prove the lower bound, we use the potential method (cfr.~\Cref{potential}). Let $\mu$ be the mass distribution on $\Gamma_T$ defined as
	$$ \mu(A) = \mathcal{L} \tonde{ \s{ x\in \S^d \; \big\vert\; (x,f(x)) \in A} } \ , $$
	where $\mathcal L$ is the Lebesgue measure on $\S^d$. Let $s = d+1-\beta$,  using Fubini's theorem we obtain
	$$A_s= \E\quadre{\iint_{\Gamma_T^2}  d(z,w)^{-s} \di \mu(z) \di \mu(w)} =\iint_{(\S^d)^2}\E\quadre{ \tonde{ |T(x)-T(y)| \vee  d_{\S^d}(x,y)}^{-s} }\di x \di y \ .
	$$
	If $x,y\in \S^d$, we denote with $\theta_{x,y}$ the angle between $x$ and $y$.  We split the proof in two parts, one for $\beta\in(0,1)$ and one for $\beta=1$.
	\paragraph*{$\beta\in(0,1)$}
	By~\Cref{teo7}, there exists $\theta_0>0$ such that, for any $x,y$ such that $\theta_{x,y}<\theta_0$, we have
	\begin{equation}\label{qq0} |T(x)- T(y)| \vee \theta_{x,y} \geq |T(x)-T(y)| \geq \frac{K} 2\theta_{x,y}^{\beta} \sqrt{|\ln(\theta_{x,y})| }
	\end{equation}
	with probability one.  Furthermore,
	\begin{equation}
		\label{qq1}|T(x)- T(y)| \vee \theta_{x,y}  \geq \theta_{x,y} \ .
	\end{equation}
	Combining~\eqref{qq0} and \eqref{qq1},  using spherical coordinates, since $-s\beta + d-1 > -1 $ we get
	$$ A_s \leq \frac{\w_d K} 2 \int_{0}^{\theta_0} \theta^{-s\beta} |\ln(\theta)|^{-s/2} \theta^{d-1} \di \theta + \omega_d \int_{\theta_0}^\pi  \theta^{-s+d-1} \di \theta <\infty  \  .$$
	Thus,
	$$ \E\quadre{\iint_{\Gamma_T^2} d(z,w)
		^{-(d+1-\beta)}\di \mu(z)\di \mu(w) } < \infty  $$
	and hence
	$$ \iint_{\Gamma_T^2} d(z,w)
	^{-(d+1-\beta)}\di \mu(z)\di \mu(w)  < \infty$$
	almost surely in $\W$. By~\Cref{potential} we finally obtain
	$$ \dH(\Gamma_T) \geq d+1 - \beta  \ . $$
        \paragraph*{$\beta=1$}
In view of \eqref{qq1}, for every $\varepsilon \in (0,d) $ we have
		$$A_{d-\varepsilon} \leq  \int_0^\pi  \frac{\di \theta}{\theta^{1-\varepsilon}}<\infty \ . $$
		Using again~\Cref{potential} we have
		$$ \dH(\Gamma_T) \geq d -\varepsilon \ . $$
		Thus
		$$ d-\varepsilon \leq  \dH(\Gamma_T) \leq d $$
		for every $ \varepsilon > 0 $, whence the claim.
	\subsection{Proof of \eqref{almost}: Hausdorff dimension of the excursion set}\label{pp111} For any $u\in \R$, let
	$$ L_u = \s{ (x,u) \; | \; x \in \S^d  } . $$
	Then
	$$\dH(T^{-1}(u)) = \dH(\Gamma_T\cap L_u) \ .$$ To obtain the lower bound we rely again on the potential method.
The argument is somewhat technical and closely follows the approach in~\cite{Adler}. For completeness, we provide the full derivation in~\Cref{App_HAU}. More precisely, we show that, for any fixed $u \in \mathbb{R}$, there exists a (random) measure $\mu$ such that, for every $\gamma < d - \beta$,
		\[
		\int_{T^{-1}(u) \times T^{-1}(u)} \frac{1}{\|z-w\|^{\gamma}} \, \mathrm{d}\mu(z)\,\mathrm{d}\mu(w) < \infty \qquad \text{a.s.}
		\]
		and that, with positive probability, $\mu$ is a mass distribution supported on $T^{-1}(u)$.
	Applying~\Cref{potential}, we conclude that
	\[
	\dim_{\mathrm{H}}(\Gamma_T \cap L_u) \;\ge\; d - \beta \, .
	\]
	For the upper bound one can use~\Cref{cor_updim}, which completes the proof of~\Cref{teo_fractal}.

	\section{Proof of \Cref{th0}: Kac-Rice class implies $C^1$}

	Let $\varphi_N:\, \S^d \setminus \s{N} \to \R^d$ denote the stereographic projection from the north pole $N$  and $\varphi_S:\, \S^d \setminus \s{S} \to \R^d$ the one from the south pole $S$. Then,
	\begin{align*}
		&\varphi_N\tonde{\begin{pmatrix} x_1 , \dots , x_{d+1} \end{pmatrix}} = \begin{pmatrix}
			\frac{x_1}{1-x_{d+1}} , 	\frac{x_2}{1-x_{d+1}} , \dots , \frac{x_d}{1-x_{d+1}}
		\end{pmatrix} \ , \\
		&	\varphi_S\tonde{\begin{pmatrix} x_1 , \dots , x_{d+1} \end{pmatrix}} = \begin{pmatrix}
			-\frac{x_1}{1+x_{d+1}} , 	\frac{x_2}{1+x_{d+1}}, \dots ,  \frac{x_{d}}{1+x_{d+1}}
		\end{pmatrix} \ .
	\end{align*}
	Since $ \{ (\S^d \setminus \s{N}, \varphi_N), (\S^d \setminus \s{S}, \varphi_S) \} $ is an atlas of $\S^d$, the fact that $ T \in C^1(\S^d) $ almost surely is equivalent to the fact that $f_N = T \circ \varphi_N^{-1}$ and $f_S = T \circ \varphi_S^{-1}$ are in $C^1(\R^d)$ almost surely. The main steps of the proof are as follows:

	\begin{enumerate}[label = \roman*)]
		\item \label{it:i}First,  we prove that $f_N$ has a mean-square derivative in every direction: that is, if $e_i$ is the $i$-th vector of the canonical basis of $\R^d$, then there exists a random field $D^i f_N$ such that, for every $s \in \R^d$,
		\begin{equation}\label{derivata}
			\lim_{h \to 0} \frac{f_N(s+he_i) - f_N(s)}{h} = D^i f_N(s) \ ,
		\end{equation}
		where the limit holds in $L^2(\W)$.
		\item \label{it:ii} Secondly, we prove that, for every $s \in \R^d$ and $\varepsilon > 0$,
		\begin{equation}\label{somma}
			\sum_{n=1}^\infty \P\tonde{\left| n \tonde{f_N(s+1/n) - f_N(s)} - D^i f_N(s) \right| > \varepsilon} < \infty \ ,
		\end{equation}
		and therefore, by the Borel--Cantelli lemma, the limit in~\eqref{derivata} holds also almost surely.

		\item \label{it:iii} Then, we prove that there exist $\eta,\zeta_1, \zeta_2, K > 0$ such that
		\begin{equation}\label{kol}
			\E\quadre{| D^i f_N(s) - D^i f_N(t)|^{\zeta_1}} \leq K \n{t-s}^{d+\zeta_2} \qquad \forall t, s \in \R^d, \ \ \n{t-s} < \eta .
		\end{equation}
		Thus, by Kolmogorov’s continuity theorem \cite{kallenberg2002foundations}, Theorem 3.23, there exists a continuous version of $D^i f_N$.

		\item \label{it:iv}Finally, by reviewing the proofs of items \ref{it:i}-\ref{it:iii}, we observe that the arguments rely solely on the expression for the covariance function $\Sigma_{f_N} $ of the field $f_N$,
		$$
		\Sigma_{f_N}(s,t) = \E[f_N(s)f_N(t)] \ .
		$$
		Since $\Sigma_{f_N} = \Sigma_{f_S}$, the results established for $f_N$ also hold for $f_S$.
	\end{enumerate}

	We will also make use of the following standard lemmas.

	\begin{lemma}\label{tecnico1} Let $(X_n)_{n\in \N}$ be a sequence of random variables with finite second moment. Then
		$$ X_n \xrightarrow{L^2} X \quad \Longleftrightarrow \quad  \lim_{n, m\to + \infty} \E[X_n X_m] <\infty  \ .$$
	\end{lemma}

	\begin{lemma}\label{tecnico2} Let $(X_n)_{n\in \N}$ be a sequence of random variables with finite second moment and such that $X_n \xrightarrow{L^2}X$.
		Then, for any $Y$ with finite second moment,
		$$ X_n Y\xrightarrow{L^2}XY \ . $$
	\end{lemma}

	\subsection{Proof of item \ref{it:i}}

	Let us abbreviate $ f_N = f $. From Lemma~\ref{tecnico1}, we know that~\eqref{derivata} holds in $L^2$ if and only if the following limit is finite:
\begin{equation}\label{limi}
	\begin{aligned}
		\lim_{h,k \to 0}  A_{h,k}
		\quad &\text{where}\\
		A_{h,k} &=  A_{h,k}(s,s) = \frac{1}{hk}
		\E\!\left[\big(f(s + h e_i) - f(s)\big)\big(f(s + k e_i) - f(s)\big)\right].
	\end{aligned}
\end{equation}
	A simple computation shows that, for all $ s \in \R^d $,
	$$
	\varphi_N^{-1}(s) = \left(
	\frac{2s}{1 + \nn s} , \frac{\nn s - 1}{1 + \nn s}
	\right) \in \R^{d+1} \ ,
	$$
	and hence, for all $ s , t \in \R^d $,
	\begin{align*}
		\langle \varphi_N^{-1}(s), \varphi_N^{-1}(t) \rangle = & \frac{4 \langle s, t \rangle + (\nn s - 1)(\nn t - 1)}{(1 + \nn s)(1 + \nn t)}
		= 1 - \frac{2 \nn{s - t}}{(1 + \nn s)(1 + \nn t)} \ .
	\end{align*}
	Thus, for all $ s , t \in \R^d $ we have
	\begin{equation}\label{covarianza}
		\Sigma_{f_N}(s, t) = \E[T(\varphi_N^{-1}(s)) T(\varphi_N^{-1}(t))] = \kappa\tonde{1 - \frac{2 \nn{s - t}}{(1 + \nn s)(1 + \nn t)}} \ .
	\end{equation}
	Set
	$$ E(x,y)  = \frac{2(x-y)^2}{(1+\nn{s + xe_i}) (1+\nn{s + ye_i})  }\,.$$
	Since for every $x,y$, $E(x,y) \geq 0 $ and $1-E(x,y) \to 1 $ as $x,y\to 0 $, there exists $\delta>0$ such that
	$$ \frac{1}2 \leq  1-E(x,y)\leq 1 \qquad \forall (x,y) \in B(0, 3\delta) \ , $$
	where $B(u,r)$ denotes the open ball in $\R^2$ with center $u\in \R^2$ and radius $r$. Thus, the function
	$$ G:\, [-\delta, \delta]^2 \to [-1,1] \qquad G(x,y) = \kappa(1-E(x,y)) $$
	is well defined. Let $h,k\in \R$ such that $|h|, |k|\leq \delta$. Then
	\begin{align*}A_{h,k} =\frac{ G(h,k) - G(h,0) - G(0,k) + G(0,0)}{hk}\,.
	\end{align*}
	For every $k$, we define
	$$ \phi_{hk}:\, [0,1] \to \R \qquad  t\mapsto  G(h, tk ) - G(0,tk) \ . $$
	Then
	$$
	 A_{h,k}  = \frac{\phi_{hk}(1) - \phi_{hk}(0)}{hk} \ .
	$$
	Since $\kappa\in C^1([-1,1])$ and $E\in C^\infty(\R^2)$, we have $\phi_{hk}\in C^1([0,1])$ and hence
	$$ \phi_{hk}(1) - \phi_{hk}(0)  = \int_{0}^1 \phi_{hk}'(t) \di t  =k \int_0^1\Big(  \partial_2 G(h,tk) - \partial_2 G(0,tk)\Big)  \di t    \  . $$
	We set
	$$ \psi_{y}(t) : [-\delta,\delta]\to \R \qquad  t \mapsto  \partial_2 G(t,y) \ . $$
	In \Cref{gruffalo} we prove that $\psi_y$ is Lipschitz and therefore absolutely continuous.
	Hence its derivative exists almost everywhere, and it is therefore almost everywhere equal to $\partial_{12} G(\cdot,y)$. Then
	$$ A_{h,k}   = \frac{1}{h} \int_0^1 \int_0^h \partial_{12} G(r,tk) \di r   \di t    =  \int_0^1 \int_0^1 \partial_{12} G(rh,tk) \di r   \di t    \  . $$
	Using \eqref{supM} and \eqref{limpsi} (with $y=0$), by dominated convergence we obtain
	\begin{equation}\label{hkto0}\lim_{h,k\to 0 } A_{h,k}  =    \frac{4\kappa'(1)}{(1+\nn s)^2}\,.
		\end{equation}
		We have thus proved \ref{it:i}.
	\begin{lemma} \label{gruffalo} The function $\psi_y$ is Lipschitz for every $y$.
	\end{lemma}
	\begin{proof}
		A simple computation shows that
		\begin{equation}\label{der2}
			\begin{aligned}\partial_{12} G(x,y) = & \kappa''(1- E(x,y)) \partial_1 E(x,y) \partial_2 E(x,y) - \kappa'(1-E(x,y) )\partial_{12}E(x,y) \,.
			\end{aligned}
		\end{equation}
		Now, $E\in C^\infty(\R^2)$ and $\kappa\in C^1([-1,1])$, thus there exists $M_1>0$ such that
		\begin{equation}\label{bb1} \sup_{x,y\in [-\delta, \delta]} \Big| \kappa'(1-E(x,y) )\partial_{12}E(x,y)  \Big| = M_1\, .
		\end{equation}
		Define
		$$M(x,y)  = \kappa''(1-E(h+y,y)) \partial_1 E(x,y)  \partial_2 E(x,y) \ . $$
		A trivial computation shows that, as $h \to 0 $,
		\begin{align*}M(y+h,y) =- \frac{16h^2 \kappa''(1-E(h+y,y))}{(1+\nn{s+ye_i})^2 (1+\nn{s+(h+x)e_i})} + O(h^3)\,.
		\end{align*}
		Hence, using the definition of CRI,
		\begin{align*}M(y+h,y)  =&-\frac{ 16c_1 \beta(\beta-1)h^2}{(1+\nn {s+(h+y)e_i})^2 (1+\nn{s+ye_i})^2} \times \\
			& \quad
			\times \Bigg[ \Bigg( \frac{2h^2}{(1+\nn{s+(y+h)e_i})(1+\nn{s+ye_i})}\Bigg)^{\beta-2} \hspace{-10pt} + o(h^{\beta-2})\Bigg]+ O(h^3)\\
			&= O\Big( h^{2(\beta-1)\wedge 3}\Big)\,.
		\end{align*}
		Since $\beta>1$, there exists $\varepsilon>0$ such that
		\begin{equation}\label{bb2}\sup_{x,y\in [-\delta, \delta] \atop{|x-y|<\varepsilon}} |M(x,y)|< 1 \,.
		\end{equation}
		Otherwise, for $|x-y|>\varepsilon$, $E(x,y)< 1$. Since $\kappa\in C^2((-1,1))$, there exists $M_2>0$ such that
		\begin{equation}\label{bb3}\sup_{x,y\in [-\delta, \delta] \atop{|x-y|\geq \varepsilon}} |M(x,y)| = M_2 \,.
		\end{equation}
		Combining~\eqref{der2} with~\eqref{bb1},~\eqref{bb2} and~\eqref{bb3}, we get
		\begin{equation} \label{supM}
		 \sup_{x,y\in [-\delta, \delta]} \Big| \partial_{12} G(x,y)\Big|\leq M_1 + (M_2 + 1) = M \, .
		 \end{equation}
		The previous computation shows also that $\psi'_y(x)$ is well-defined and coincides with $\partial_{12} G(x,y)$ for every $x\neq y$. Moreover,
		\begin{equation} \label{limpsi}
		\lim_{x\to y} \psi_y'(x) = - \kappa'(1) \partial_{12} E(x,y)\Big|_{x=y} = \frac{4\kappa'(1)}{(1+\nn{s + ye_i})^2} \leq M_1 \,.
		\end{equation}
		This implies that $\psi_y$ is $M$-Lipschitz.
	\end{proof}
\subsection{Covariance computation}
Let $s,t\in \R^d$ and $i,j\in 1, \dotsc,d$. For every $h,k\in \R$, define
	\begin{align*}
		A_{h,k}(s, t; i, j) = \E \quadre{\frac{f(t + he_i) - f(t)}{h} \cdot \frac{f(s + ke_j) - f(s)}{k}} \ .
	\end{align*}
Note that this definition coincides with the one given in~\eqref{limi}, in the case $s = t$ and $i = j$.
	\begin{align*}
		& x_{h,k}(s,t) =  1-\frac{2\nn{t-s + he_i - k e_j}}{(1+\nn{t+h e_i})(1+\nn{s+ke_j})} \ ,
	\end{align*}
	and write  $ x_{0,0}(s,t) = x(s,t) $.
	From~\eqref{covarianza}, using Lagrange's theorem we obtain
	\begin{align*}
		A_{h,k}(s,t;i,j)= &\frac{1}{hk} \Big[ \kappa(x_{h,k}(s,t))- 	\kappa (x_{0,k}(s,t))- 	\kappa (x_{h,0}(s,t)) + 	\kappa(x(s,t)) \Big] \\
		= &\frac{1}{hk} \Big[ \kappa'(\xi_1) \Big(x_{h,k}(s,t) - x_{0,k}(s,t)\Big) - \kappa'(\xi_2)\Big(x_{h,0}(s,t) - x(s,t) \Big) \Big]\\
		= &\frac{1}{hk} \Bigg[2 \kappa'(\xi_1)  \frac{ \nn{t-s -ke_j} ( 2t_i h +h^2)-(1+\nn t )(2h(t_i-s_i-k\delta_{i,j})+h^2)}{(1+\nn{t})(1+\nn{t+he_i})(1+\nn {s+ke_j})} \\
		& -2 \kappa'(\xi_2)   \frac{\nn{t-s} (2t_ih+h^2)
			-(1+\nn t )( 2h(t_i-s_i)+h^2)
		}{(1+\nn t )(1+\nn{t+he_i})(1+\nn s)}  \Bigg] \ ,
	\end{align*}
	where $ \xi_1 = \xi_1(s,t;h,k) $ belongs to the interval with endpoints $x_{h,k}(s,t)$ and $x_{0,k}(s,t)$,
	and $ \xi_2 = \xi_2(s,t;k) $ to the interval with endpoints $x_{h,0}(s,t)$ and $x(s,t)$. Using the continuity of $\kappa'$, 	defining $ A_k(s,t;i,j) =\lim_{h \to 0 } A_{h,k} (s,t;i,j) $,  we have
	\begin{equation}\label{limite}
		\begin{aligned}A_k(s,t;i,j)
			& = \frac{1}{k} \Bigg[4 \kappa'(x_{0,k}(s,t)) \tonde{ \frac{ t_i\nn{t-s-ke_j}  -(t_i-s_i-k\delta_{i,j})(1+\nn t)}{(1+\nn{t})^2(1+\nn {s+ke_j})} }\\
			& \phantom{=} -4 \kappa'(x(s,t))\tonde{  \frac{t_i\nn{t-s}
					- (t_i-s_i) (1+\nn t)
				}{(1+\nn t )^2(1+\nn s)} } \Bigg]\\
			& = \frac 1 k\Bigg[4 \frac{t_i\nn{t-s}
				- (t_i-s_i) (1+\nn t)}{(1+\nn t )^2 (1+\nn{s+ke_j})
			}\Big(\kappa'(x_{0,k}(s,t)) - \kappa'(x(s,t)) \Big) \Bigg] \\
			& \phantom{=} +4\frac{t_i(k-2(t_i - s_i)) + \delta_{i,j}(1+\nn t) }{(1+\nn t )^2 (1+\nn{s+ke_j})}\kappa'(x_{0,k}(s,t))	\\
			& \phantom{=} - 4\frac{ t_i \nn {t-s} -(t_i-s_i)(1+\nn t )}{(1+\nn t)^2 (1+\nn s)(1+\nn{s+ke_j})}(2s_i + k)\kappa'(x(s,t)) \ .
		\end{aligned}
	\end{equation}
	If $t=s$, then
	\begin{align} \label{eq:Akttij}
		A_k(t,t;i,j)& = 4 \frac{k t_i + \delta_{i,j} (1+\nn t ) }{ (1+\nn t)^2(1+\nn{t+ke_j})} \kappa'(x_{0,k}(t,t)) \ .
	\end{align}
	whence, from the continuity of $\kappa'$,  we obtain
	\begin{equation}\label{cov_der}\lim_{k\to 0 } A_k(t,t;i,j) = \frac{4\kappa'(1)}{(1+\nn t)^2} \delta_{i,j}\ .
	\end{equation}

	\subsection{Proof of item \ref{it:ii}} The triple $(f(s), f(t), D^if(s))$ is a multivariate Gaussian vector. Indeed, for all $a_1,a_2, a_3\in \R$,
	we have that $ a_1 f(s) + a_2 f(t)+ a_3 D^i f(s) $ is the $L^2$-limit of a Gaussian random variable:
	$$ a_1 f(s) + a_2 f(t)+ a_3 D^i f(s)= \lim_{h\to 0 }\tonde{
		\tonde{a_1 -\frac{a_3} h}  f(s) + a_2 f(t) +\frac{a_3} h  f(s + he_i)} \ .
	$$
	Let us call
	$$ \sigma_n ^2 =  \Var(n (f(s+e_i/n) - f(s)) -D^i f(s)) \ . $$
	The random variable $n (f(s+e_i/n) - f(s)) -D^i f(s)$ is distributed as $\sigma_n Z$ for $Z\sim \mathcal N(0,1)
	$, and hence
	\begin{align*}\label{bc}
		\P\tonde{| n (f(s+e_i/n) - f(s)) -D^i f(s)|>\varepsilon
		} \leq  2 \P\tonde{Z> \frac{\varepsilon}{\sigma_n}}  \leq \frac{2\sigma_n}{\sqrt{2\pi}\varepsilon} \mathrm e^{-\frac{\varepsilon^2}{2\sigma_n^2}} \  .
	\end{align*}
	If we prove that $ \sigma_n ^2 = O(n^{-\zeta}) $ for some $\zeta> 0$, then we will obtain~\eqref{somma}.  To simplify the notation, we set  $\frac{1}{n} = h$. Then
	\begin{equation}\label{derivata2}
		\begin{aligned} \sigma_{1/h}^2 = &\frac{1}{h^2} \E\quadre{\big( f(s+h e_i)- f(s)\big)^2} + \E[ D^i f(s)^2]
			-\frac{2} h\E[ \big(f(s+h e_i) -f(s)\big) D^i f(s)] \ .
		\end{aligned}
	\end{equation}

	The expectation in the first term is equal to
	\begin{align} \label{eq:si1}
		& \ \E\quadre{\big( f(s+h e_i)- f(s)\big)^2} \nonumber \\
		= & \ \E[ f(s+h e_i)^2] +\E[ f(s)^2]-2\E[f(s+ he_i) f(s)] \nonumber \\
		= & \ 2\kappa(1) - 2\kappa\tonde{1- \frac{2h^2}{(1+\nn s)(1+\nn {s+he_i})}} \nonumber \\
		\nonumber= & \  \kappa'(1) \frac{4h^2}{(1+\nn s)(1+\nn {s+he_i})} -\frac{4ch^{2\beta}}{(1+\nn s)^\beta (1+\nn{s+he_i})^\beta} +o(h^{2\beta}) \ .
	\end{align}
	In  the second-last equality we have used~\eqref{covarianza}, while  the last equality relies on~\eqref{dec::bb}, since in the Kac-Rice class $p_1(t) = 1 -\kappa'(1) t $. To ease the forthcoming calculation we have set $c= 2^{\beta-1}c_1$.

	The second term in~\eqref{derivata2} can be written, using~\Cref{tecnico2} and~\eqref{cov_der}, as
	\begin{equation*}\label{var_der}\E[D^i f(s)^2] = \lim_{k \to 0 } \lim_{h \to 0 } A_{h,k}(s,s;i,i)  = \frac{4\kappa'(1)}{(1+\nn s)^2} \ .
	\end{equation*}

	Finally to compute the third term in~\eqref{derivata2} one can use~\Cref{tecnico2}, \cref{eq:Akttij} and the formula for $\kappa'$ obtained by deriving~\eqref{dec::bb} to get that
	\begin{align*}
		& \frac 1 h\E[\big( f(s+he_i)-f(s)\big) D^i f(s)] = \lim_{k \to 0} A_{k,h} (s,s;i,i) \nonumber \\
		= & \ \frac{  4(1+\nn s) +4 h s_i}{(1+\nn s)^2 (1+\nn{s+he_i})} \kappa'\tonde{  1- \frac{ 2h^2}{(1+\nn s)(1+\nn{s+he_i})}} \nonumber \\
		= & \ \frac{ 4(1+\nn s) +4 h s_i + }{(1+\nn s)^2 (1+\nn{s+he_i})}\tonde{  \kappa'(1) - \frac{c\beta h^{2\beta-2}}{(1+\nn s)^{\beta-1}(1+\nn{s+he_i})^{\beta-1}}+o(h^{2\beta-2})} \ .
	\end{align*}

	Putting back the calculations for three summands into~\eqref{derivata2} yields
	\begin{align*} \sigma_{1/h}^2 =&  \frac{4\kappa'(1)}{(1+\nn s)^2(1+\nn{s+he_i})} h^2
		+\frac{4c(2\beta-1)}{(1+\nn s)^\beta (1+\nn{s+he_i})^\beta}h^{2\beta -2} + o(h^{2\beta-2}) \ .
	\end{align*}
	Setting
	\begin{align*}
		& C(\beta)  =\frac{4\kappa'(1)}{(1+\nn s)^2(1+\nn{s+he_i})}\mathbbm
		{1}_{\s{ 2}}(\beta) + \frac{4c(2\beta-1)}{(1+\nn s)^\beta(1+\nn{s+he_i})^\beta}
		\mathbbm{1}_{(1,2)}(\beta) \\
		& \rho(\beta) =1 \wedge \beta -1  \ ,
	\end{align*}
	we finally obtain
	$$ \sigma_n^2 = C(\beta) n^{-2\rho(\beta)} +o(n^{-2\rho(\beta)})  $$
	hence the claim holds, taking $\zeta = \rho(\beta)$.

	\subsection{Proof of item \ref{it:iii}}
	We first prove that there exists $ C > 0 $ such that
	\begin{equation}\label{mom1}\E[|D^i f(s) - D^if(t)|^2] \leq C	 \n{t-s}^{2\rho(\beta)} \ .
	\end{equation}
	Indeed, we have
	$$ D^i f(s) - D^i f(t) = \lim_{h\to 0 } \lim_{k\to 0 } \frac{1}{hk} \Big[  h f(s+ke_i) - hf(s) - kf(t+he_i)+ kf(t) \Big]$$
	where the limits hold  in the $L^2$-sense. As before, $D^i f(s) - D^i f(t)$ is a Gaussian random variable and hence
	\begin{align*} \E[ |D^i f(s) - D^i f(t)|^{2k}] &=\E[ |D^i f(s) - D^i f(t)|^{2}]^k  (2k-1)!!\\
		&\leq C^k (2k-1)!! \n{t-s}^{2k\rho(\beta)}\ ;
	\end{align*}
	as a consequence, taking   $k = d$ we have $2d\rho(\beta) \leq d+d $ and therefore~\eqref{kol} holds with $\zeta_1= 2d$, $\zeta_2= d$ and  $K = C^k (2k-1)!!$.\\

	From~\eqref{var_der}, we can compute $\E[D^i f(s)^2]$ and  $\E[D^i f(t)^2]$, whence to estimate~\eqref{mom1} we only need to evaluate $\E[D^if(s)D^if(t)]$. Using again~\Cref{tecnico2} we obtain
	\begin{align*}
		C(s,t)=\E[D^if(s)D^if(t)] & = \lim_{k \to 0} \lim_{h \to 0} A_{h,k}(s,t;i;i) \ .
	\end{align*}
	Since $\kappa'$ is continuous in the closed interval with endpoints  $x_{0,k}(s,t), x(s,t)$ and differentiable in the interior of the same interval, using Mean Value Theorem, there exists  $\xi_k(s,t)$ in the same intervals such that
	\begin{align*}&\kappa'(x_{0,k}(s,t)) - \kappa'(x(s,t)) \\
		&= \kappa''(\xi_k(s,t)) \Big(
		x_{0,k}(s,t) - x(s,t)\Big) \\
		&= 2\kappa''(\xi_k(s,t))  \frac{\nn{t-s}( k^2 + ks_j) - ( k^2 - 2k(t_j-s_j))(1+\nn s)}{(1+\nn t )(1+\nn s)(1+\nn{s+ke_j})}
	\end{align*}
	and therefore~\eqref{limite} becomes
	\begin{align*}
		&A_k(s,t)\\
		& = 8 \frac{t_i\nn{t-s}
			- (t_i-s_i) (1+\nn t)}{(1+\nn t )^2 (1+\nn{s+ke_j})
		}\cdot \frac{\nn{t-s}( 2s_j+k ) - ( k- 2(t_j-s_j))(1+\nn s)}{(1+\nn t )(1+\nn s)(1+\nn{s+ke_j})} \kappa''(\xi_3)\\
		&+4\frac{t_i(k-2(t_i - s_i)) + 1+\nn t }{(1+\nn t )^2 (1+\nn{s+ke_j})}\kappa'(x_1)	- 4\frac{ t_i \nn {t-s} -(t_i-s_i)(1+\nn t )}{(1+\nn t)^2 (1+\nn s)(1+\nn{s+ke_j})}(2s_i + k) \ .
	\end{align*}
	Therefore
	\begin{align*}
		&\E[D^if(s)D^jf(t)]= \lim_{k\to 0}	A_k(t,s;i,j) \\
		=&\frac{16 \tonde{ t_i\nn{t-s}
				- (t_i-s_i) (1+\nn t)} \tonde{s_i\nn{t-s} +(t_i-s_i)(1+\nn s)} }{(1+\nn t )^3 (1+\nn{s})^3\kappa''(x_1)
		}\\
		&+8\kappa'(x_2)\tonde{\frac{ 1+\nn t  -t_i(t_i - s_i)}{(1+\nn t )^2 (1+\nn{s})}- \frac{s_i t_i \nn {t-s} -s_i(t_i-s_i)(1+\nn t )}{(1+\nn t)^2 (1+\nn s)^2}} \ .
	\end{align*}
	Thus
	\begin{align*}
		C(s,t) = &
		4 \Bigg( \frac{\kappa'(1)}{(1+\nn t)^2} +   \frac{\kappa'(1)}{(1+\nn s)^2} - 4\kappa'(x_2)B_1(s,t)  - 32 B_2(s,t) \kappa''(x_2)
	\end{align*}
	where
	$$ B_1(s,t) =  \frac{ 1+\nn t  -t_i(t_i - s_i)}{(1+\nn t )^2 (1+\nn{s})}- \frac{s_i t_i \nn {t-s} -s_i(t_i-s_i)(1+\nn t )}{(1+\nn t)^2 (1+\nn s)(1+\nn{s})}$$
	and
	$$B_2(s,t)= \frac{
		\tonde{t_i\nn{t-s}
			- (t_i-s_i) (1+\nn t) }\tonde{s_i\nn{t-s} +(t_i-s_i)(1+\nn s)}}{(1+\nn t )^3 (1+\nn{s})^3
	}
	$$
	Now note that
	\begin{align*}&B_1(s,t) \\
		=&\frac{1}{(1+\nn t )(1+\nn s)}\\
		& + \frac{1}{(1+\nn t)^2 (1+\nn s)^2}
		\Bigg( (t_i-s_i) \tonde{  t_i (1+\nn s)  - s_i (1+\nn t )} + s_i t_i \nn {t-s} \Bigg)\\
		= & \frac{1}{(1+\nn t )(1+\nn s)} \\
		&+ \frac{4 (t_i-s_i)^2 (1+\nn s) + 4s_i(t_i-s_i)(\nn s - \nn t)+4 s_i t_i \nn {t-s}}{(1+\nn t)^2 (1+\nn s)^2}\\
		= & \frac{1}{(1+\nn t )(1+\nn s)} + \widetilde B_1(t,s) \ .
	\end{align*}
	Now,
	\begin{align*}|\widetilde B_1(s,t)| \leq  &  \frac{4 \nn{t-s} (1+\nn s) + 4|s_i|\n{t-s}\big| \nn s - \nn t \big| +4 |s_i t_i|  \nn {t-s}}{(1+\nn t)^2 (1+\nn s)^2}\\
		\leq &  \frac{4 \nn{t-s} (1+\nn s) + 4(1+\nn s)^2\nn{t-s}(1+\nn t ) +4 (1+\nn s)(1+\nn t)  \nn {t-s}}{(1+\nn t)^2 (1+\nn s)^2}\\
		\leq & 4 \nn{t-s} \tonde{ \frac{1}{(1+\nn s)(1+\nn t)^2} + \frac{1 }{1+\nn t}	 +\frac 1 {(1+\nn s) (1+\nn t)}}
		\leq  12 \nn{t-s}
	\end{align*}
	where we have used the three following inequalities:
	\begin{equation}\label{diss}
		\begin{aligned}\big|  \nn s - \nn t \big|  &=\big( \n s + \n t \big)\big| \n s - \n t \big|
			\leq \n{s-t}(1+\nn s )(1+\nn t )  \ , \end{aligned}
	\end{equation}
	$$ (t_i -s_i)^2 \leq \nn{t-s}$$
	and
	$$s_i\leq |s_i| \leq \n s_\infty \leq  \n s \leq  1+ \nn s   \ .$$
	Moreover,
	\begin{align*} |B_2(s,t)| = & \frac{1}{(1+\nn t)^3(1+\nn s)^3}\Big|
		t_i s_i \n{t-s}^4 - (t_i-s_i)^2 (1+\nn s)(1+\nn t) \\
		&- s_i \nn{t-s}(t_i-s_i)(1+\nn t) + t_i(t_i-s_i)(1+\nn s )\nn{t-s}
		\Big|\\
		\leq& \frac{1}{(1+\nn t)^3(1+\nn s)^3}\bigg(
		|t_i s_i| \n{t-s}^4 +|t_i-s_i|^2 (1+\nn s)(1+\nn t) \\
		&+ \nn{t-s}|t_i-s_i| \Big| - s_i(1+\nn t) + t_i(1+\nn s )\Big| \bigg)\ .
	\end{align*}
	With the same argument as  we have used for $B_1(s,t)$ we have also
	$$ |B_2(s,t)|\leq \nn{t-s}+ 3 \n{t-s}^4  \ . $$
	Using~\eqref{dec::bb} we obtain
	\begin{align*}
		&\frac 1 4 C(s,t) \\
		=& \frac{\kappa'(1)}{(1+\nn t)^2} +  \frac{\kappa'(1)}{(1+\nn s)^2}\\
		& - 4 \tonde{ \kappa'(1) - c\beta 2^{\beta-1} \frac{\n{t-s}^{2\beta-2}}{(1+\nn s)^\beta(1+\nn t)^\beta} +o(\n{t-s}^{2\beta-2})} \tonde{\frac 1{(1+\nn s)(1+\nn t )} + B^1(s,t)}\\
		&-32 B_2(s,t)\tonde{ c \beta (\beta-1) 2^{\beta-2} \frac{\n{t-s}^{2\beta-4}}{(1+\nn s)(1+\nn t)^{\beta-1}} +o(\n{t-s}^{2\beta-4})}\\
		\leq &	|\kappa'(1)| \Bigg| \frac{1}{(1+\nn t)^2} +  \frac{1}{(1+\nn s)^2}  - \frac{4}{(1+\nn t)(1+\nn s)}\Bigg| \\
		& + 48 \kappa'(1) \nn{t-s} +   |c|\beta 2^{\beta+1} \frac{\n{t-s}^{2\beta-2}}{(1+\nn s)^{\beta+1}(1+\nn t)^{\beta+1} } +o\tonde{ \n{t-s}^{2\beta-2}}
		\\
		&+32 \tonde{ c \beta (\beta-1) 2^{\beta-2} \frac{\n{t-s}^{2\beta-2}}{(1+\nn s)(1+\nn t)^{\beta-1}} + o(\n{t-s}^{2\beta-2})} \ .
	\end{align*}
	Since $1<\beta<2$ we have also $\nn{t-s} = o(\n{t-s}^{2\beta-2})$ and therefore
	\begin{align*}C(s,t) \leq &
		4|\kappa'(1)| \Bigg| \frac{1}{(1+\nn t)^2} +  \frac{1}{(1+\nn s)^2}  - \frac{4}{(1+\nn t)(1+\nn s)}\Bigg| \\
		& +    |c| \beta \tonde{4\beta-3}2^{\beta+3}\n{t-s}^{2\beta-2} + o(\n{t-s}^{2\beta-2})\ .
	\end{align*}
	To conclude, it is sufficient to note that
	\begin{align*}&\frac{1}{(1+\nn t)^2} +  \frac{1}{(1+\nn s)^2}  - \frac{4}{(1+\nn t)(1+\nn s)}\\
		&= \frac{(1+\nn t)^2 + (1+\nn s)^2  -4(1+\nn s)(1+\nn t)}{(1+\nn t)^2 (1+\nn s)^2}\\
		& =\frac{ \Big( (1+\nn t) - (1+\nn s)\Big)^2  -2(1+\nn s)(1+\nn t)}{(1+\nn t)^2 (1+\nn s)^2}\\
		&\leq \frac{ \Big( \nn t - \nn s)\Big)^2  }{(1+\nn t)^2 (1+\nn s)^2} \leq \nn{t-s}
	\end{align*}
	where the last inequality follows by~\eqref{diss}. Summing up, we have proved that there exists $C_\beta$ such that
	$$ C(s,t) \leq C_\beta \n{t-s}^{2\rho(\beta)}  + o(\n{t-s}^{2\rho(\beta)}\leq  (C_\beta + 1)\n{t-s}^{2\rho(u)}$$
	for any $t,s$ such that $\n{t-s}<\eta$ for some $\eta$ small enough.

	\section{Proof of~\Cref{th1}: boundary volumes in the Kac-Rice class}\label{kac-rice_proof} In~\Cref{below}, we compute the expected values of the boundary volumes of a $C^1$ isotropic Gaussian random field.   \Cref{th1} follows using the fact that  $\kappa_L'(1) = \kappa'(1) ^L$ as proved in ~\cite{nostro}. We note that, using~\Cref{th0}, under the assumptions of~\Cref{th1}, the fields $T_L$ are $C^1$ a.s for every $L$.\\

	In the proof of~\Cref{below}, we use the following result that can be obtained combining Theorem .2 and Theorem 2.2 in~\cite{armentano2023generalkacriceformulameasure}.

	\begin{theorem}[See \cite{armentano2023generalkacriceformulameasure}] \label{expec} Let $X:\; \S^d \to \R$ be a Gaussian random field.
		If
		\begin{itemize}
			\item[(i)] $ X \in C^1(\S^d) $ a.s.;
			\item[(ii)] $ \Var(X(t)) > 0 $ for all $ t \in \S^d $;
		\end{itemize}
		then
		$$  \E[\mathcal{H}_{d-1} (X^{-1}(u)) ] = \int_{\S^d} \E\left. \big[ \| \nabla X(t) \| \, \right| \,X(t)= u \big] p_{X(t)}(u) \dvol(t) \ , $$
		where $p_{X(t)}$ is the density of $X(t)$ and $\nabla X(t)$ is the Riemann gradient on $\S^d$.
	\end{theorem}

	\begin{proposition}\label{below}
		Let $T$ be a Gaussian isotropic random field on $\S^d$ with zero mean and $\kappa$ as covariance function. If $ T\in C^1$ a.s. then $$\E\quadre{\mathcal H_{d-1}(T^{-1}(u))} = \w_{d-1} \kappa'(1)^{1/2} e^{-u^2/2 }\qquad a.s. \ . $$
	\end{proposition}
	\begin{proof}
		\Cref{expec} implies
		$$ \E\quadre{\mathcal{H}_{d-1}(T^{-1}(u))} = \int_{\S^d} \E\quadre{ \sqrt{\nabla T(x)^t \nabla T(x)} \; |\; T(x) = u} p_{T(x)}(u) \dvol(x)  \  . $$
		Since the field is unit-variance, isotropic and Gaussian  we have
		$$ P_{T(x)}(u) = \frac{1}{\sqrt{2\pi}} e^{-{x^2}/2}$$
		and hence
		\begin{equation} \label{EHAT}  \E\quadre{\mathcal{H}_{d-1}(T^{-1}(u))}  = \frac{1}{\sqrt{2\pi}} e^{-{u^2}/2} \int_{\S^d} \E\quadre{ \sqrt{\nabla T(x)^t \nabla T(x)} \; |\; T(x) = u} \dvol(x) \ . \end{equation}
		Let us prove that the atlas given in the proof of~\Cref{th0} is oriented. A simple computation shows that
		$$ \varphi_N \circ \varphi_S^{-1}(u) = \begin{pmatrix} -\frac{u_1}{\nn u} & \frac{u_2} {\nn u} & \cdots & \frac{u_d}{\nn u}
		\end{pmatrix}$$
		and thus, if $A = (A_{ij})$ is the Jacobian matrix of the coordinate change in $u$, we have
		$$ A_{ij} = \begin{cases}
			\frac{-\nn u \delta_{i1} +2u_i u_1 }{\n u^4}& \text{if } j=1\\
			\frac{\nn u \delta_{ij} - 2u_iu_j}{\n u^4}  & \text{if } j\neq 1
		\end{cases} \ .$$
		Instead of computing the determinant of $A$, we will compute the determinant of the matrix $\widetilde A$, which coincides with $A$ in all columns except the first one, whose sign is reversed (note that in this case, $\det(A) = -\det(\widetilde A)$). In particular, we have
		$$\widetilde A_{ij} = \frac{\nn u \delta_{ij} -2u_iu_j}{\n u^4} \ . $$
		The matrix $\widetilde A$ is diagonalizable. Its eigenvalues are $\frac{1}{\nn u}$ with multiplicity $d-1$ and $-\frac{1}{\nn u}$ with multiplicity $1$, with eigenvectors
		\begin{align*}
			& v_i = u_1 e_{i+1} - u_{i+1} e_1 \quad  i=1, \dotsc, d-1 \ ,\\
			& v_d = \sum_{i=1}^d u_i e_i \ .
		\end{align*}
		In particular, we obtain
		$$ \det(A) = - \det (\widetilde A) = \frac{1}{\n u^{2d}} \ , $$
		which implies that the atlas is oriented.\\

		Moreover, since
		\begin{align*} \frac{\partial}{\partial u_i} \varphi_N^{-1}(u) &= \frac{\partial}{\partial u_i} \begin{pmatrix} \frac{2u_1}{1+\nn u }
				& \cdots & \frac{2u_d}{1+\nn u} & \frac{\nn u -1}{1+\nn u }
			\end{pmatrix} \\
			&= \begin{pmatrix}
				\frac{2(1+\nn u) \delta_{1,i} - 4u_iu_1}{(1+\nn u)^2} & \cdots & \frac{2(1+\nn u) \delta_{d,i} - 4u_iu_d}{(1+\nn u)^2} &\frac{4u_i}{(1+\nn u)^2}
			\end{pmatrix} \ ,
		\end{align*}
		we have
		\begin{align*} & \left\langle  \frac{\partial}{\partial u_i} \varphi_N^{-1}(u), \frac{\partial}{\partial u_j} \varphi_N^{-1}(u) \right\rangle \\
			&=
			\frac{1}{(1+\nn u)^4} \tonde{ \sum_{k=1}^n (2(1+\nn u)\delta_{ik} - 4u_i u_k)(2(1+\nn u)\delta_{j,k} - 4u_ju_k)  + 16 u_i u_j} \ .
		\end{align*}
	After simple calculations, we get
	$$
	\left\langle  \frac{\partial}{\partial u_i} \varphi_N^{-1}(u), \frac{\partial}{\partial u_j} \varphi_N^{-1}(u) \right\rangle = \frac{4\delta_{ij}}{(1+\nn u)^2} \ .
	$$
	In particular, the metric tensor on the sphere is written in the coordinates induced by $\varphi_N$ (and also in $\varphi_S$) as
	$$ g_{ij} =  \frac{4\delta_{ij}}{(1+\nn u)^2} \ .$$
	Thus, the volume form (in both coordinate systems) is written as
	$$\mathrm{Vol}(u)=  \frac{2^{d}}{(1+\nn u)^d} u_1 \wedge \dotsb \wedge u_d  \ . $$
	Let  $f = T\circ \varphi^{-1}_S $;  $(D^1 f(s), \dotsc, D^d f(s))$ is a Gaussian vector. For any $\alpha_1, \dots, \alpha_d$ we have
	$$ \sum_{i=1}^d \alpha_i D^i f(s) =  \lim_{h_1, \dots,h_d \to 0} \sum_{i=1}^d \alpha_i \frac{ f(s+h_ie_i) -f(s)}{h_i} $$
	and hence, using~\eqref{cov_der},
	$$(D^1 f(u), \dotsc, D^d f(u)) \sim \mathcal N \tonde{ 0, \frac{4\kappa'(1)}{(1+\nn u)^2} I_d}   \ . $$
	Hence
	$$ \sqrt{D^1 f(s) ^2 +\dotsb+ D^d f(s)^2} \sim \frac{2\kappa'(1)^{1/2}}{1+\nn s}\chi(d) \ , $$
	where $\chi(d)$ is the square root of a $\chi^2$ random variable with $d$ degrees of freedom.
	Let us show that $f(s)$ and $D^if(s)$ are uncorrelated, and thus independent (the result is standard for random fields with $\kappa\in C^2$, but our assumptions here are more general). From~\Cref{tecnico2} we obtain
	\begin{align*} \E[f(s) D^{i} f(s) ] =& \lim_{h \to 0 } \frac{1}h \E \quadre{  \big( f(s+he_i) -f(s) \big) f(s)}
		\\= & \lim_{h\to 0} \frac{1}{h} \tonde{ \kappa\tonde{  1- \frac{2 h^2}{(1+\nn s)(1+\nn{s+he_i})}} - \kappa(1)} =0 \ ,
	\end{align*}
	since $\kappa \in C^1$, and where for the second equality we used \eqref{covarianza}.
	It follows that \begin{align*}
		\E\Big[ \sqrt{D^1 f(s) ^2 +\dotsb+  D^d f(s)^2} \ | \ f(s) =u \Big] = &\E\Big[ \sqrt{D^1 f(s) ^2 + \dotsb + D^d f(s)^2} \Big] \\
		=&  \frac{2\sqrt 2\Gamma((d+1)/2) }{\Gamma(d/2)(1+\nn s)} \kappa'(1)^{1/2} \ ,
	\end{align*}
	where we used the formula for the mean of a random variable distributed as a $\chi(d)$.

	Let $ \eta_S, \eta_N:\, \S^d \to \R$ be defined as
	$$ \eta_S(x)  = \frac{1+x_{n+1}}{2}\ , $$
	$$ \eta_N(x) = 1-\eta_S(x) = \frac{1 -x_{n+1}} 2
	\ .$$
	Clearly $(\eta_N,\eta_S)$ is a smooth partition of unity associated to our atlas. Going back to \eqref{EHAT}, we get
	\begin{align*}
		\E\quadre{\mathcal{H}_{d-1}(T^{-1}(u))}   =&\frac{e^{-u^2/2}}{\sqrt{2\pi}}\int_{\S^d}\E\quadre{ \sqrt{D^1 T(x)^2 + \cdots + D^d T(x)^2 }} \dvol(x) \\
		= &\frac{2\Gamma((d+1)/2) }{\sqrt{\pi}\Gamma(d/2)} \kappa'(1)^{1/2}e^{-u^2/2} \Bigg[
		\int_{\R^d} \eta_S(\varphi_S^{-1}(s)) \frac{1}{1+\nn s} \frac{2^d}{(1+\nn s)^d} \di s \\
		&+\int_{\R^d} \eta_N( \varphi^{-1}_N(s)) \frac{1}{1+\nn s} \frac{2^d}{(1+\nn s)^d} \di s
		\Bigg]\\
		=&\frac{2^{d+1}\Gamma((d+1)/2) }{\sqrt \pi \Gamma(d/2)} \kappa'(1)^{1/2}e^{-u^2/2}
		\int_{\R^d} \frac{1}{(1+\nn s)^{d+1}} \di s \\
		=&\frac{2^{d+1}\Gamma((d+1)/2) }{
			\sqrt \pi \Gamma(d/2)} \kappa'(1)^{1/2}e^{-u^2/2}  \frac{2 \pi^{d/2}}{\Gamma(d/2)}
		\int_{0}^{+\infty} \frac{\rho^{d-1}}{(1+\rho^2)^{d+1}} \di \rho\\
		= &  \frac{2^{d+1}\pi^{(d-1)/2}\Gamma((d+1)/2) }{\Gamma(d/2)^2} \kappa'(1)^{1/2} e^{-u^2/2} B \tonde{ \frac {d} 2, \frac{d} 2+1}\\
		= & 2^{d}\pi^{(d-1)/2}\Gamma((d+1)/2)  \frac{1}{(d-1)!} \kappa'(1)^{1/2}  e^{-u^2/2}  \ .
	\end{align*}
	From the Legendre duplication formula (see \cite{abramowitz1968handbook}, eq. (6.1.18)) we have
	$$ (d-1)! = \Gamma(d) = \Gamma(d/2) \Gamma((d+1)/2) 2^{d-1} \pi^{-1/2}$$
	and hence the claim.

\end{proof}

\section{Numerical evidence}\label{numerical}
We present a series of numerical experiments aimed at providing empirical support for the theoretical results established in this paper. These experiments shed light on the practical relevance of our findings and contribute to validating the underlying models.

All the experiments can be easily reproduced by cloning the corresponding GitHub repository: \url{https://github.com/simmaco99/SpectralComplexity}.  The simulations are run on $\S^2$ using the Healpix package~\cite{gorski2005healpix}, which is widely recognized for its efficiency in handling pixelization of the sphere and working with   in spherical harmonics. Additionally, it incorporates the Pynkowski library~(\cite{10.1093/mnras/stad3002}, \url{https://javicarron.github.io/pynkowski/pynkowski.html}), which enhances our computational framework for geometric functionals.
We generate random neural networks  with input  in $\S^2$, output in $\R$ and constant  width  $n=1000$  across the different layers; this value is more than sufficient to ensure that the Gaussian asymptotic behavior is very close to the actual realization.   In order to compute the lenght of the nodal lines, we performs a Monte Carlo estimate with $1000$ replicas.  \smallskip

Our first experiment is concerned with the fractal class. In this case, the expected value of the boundary length of an excursion set is infinity, because its  Hausdorff dimension is strictly larger than $1$. Of course, the numerical experiments will only produce finite values, because the boundary length is computed using a finite pixelization on the sphere. However we manage to illustrate the theoretical behavior by simply increasing the resolution of the experiment: as expected the numerical approximation of the boundary length grows without limit as the pixelization becomes finer and finer. This is in stark contrast with the behavior in the regular class, see~\Cref{fig::fractal}. The effect of depth is illustrated in~\Cref{fig::fractal} where the finite order approximation of the $1$-Hausdorff measure grows larger and larger as the depth increases. \\

\begin{figure}[!h]
\centering
\subfloat[ReLu]{\includegraphics[width=0.45\linewidth]{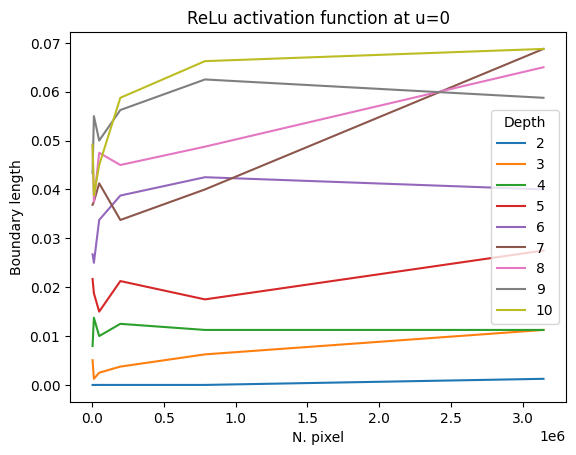}
}
\subfloat[Step]{\includegraphics[width=0.45\linewidth]{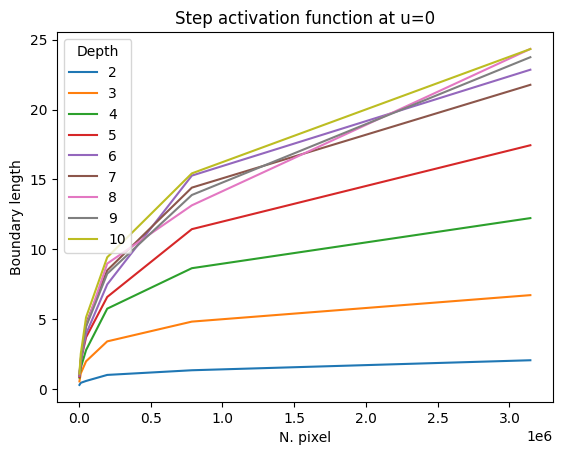}
}
\caption{Nodal length of the excursion set for random neural networks with ReLU (a) and Heaviside (b)  activation functions. As expected,  in the fractal  class (b) the length grows as the number of pixel increases.}
\label{fig::fractal}
\end{figure}

We then move to the Kac-Rice class.  For our simulations, we use  the  Gaussian activation function given by
$$ \sigma_a(x) = e^{-\frac{a}{2} x^2} \ , $$
for three different values of  $a$, that is $ 1$, $1+\sqrt 2$ and $9$.
The associated kernel is (see~\cite{nostro})
$$ \kappa^a(x) = \sqrt{\frac{ 1 + 2a}{(a + 1)^2 - (ax)^2}}$$
and so its  derivative computed at $x=1$ is given by
$$ (\kappa^a)'(1) = \frac{a^2}{2a + 1} = \begin{cases}
< 1 & \text{ if } a =1\\
= 1 & \text{ if }  a = 1+ \sqrt{2}\\
>1 & \text{ if } a = 9
\end{cases}\ . $$
This way, one can cover the low-disorder ($a=1$), sparse ($a=1+\sqrt 2 $) and high-disorder ($a=9$) regimes. \Cref{fig::nodal_length} is in accordance with Theorem~\ref{th1};   the length  of the nodal lines  exhibits different behaviors depending on the activation function: it diverges exponentially to $0$ as the depth increases in the low-disorder case, it remains constant in the sparse case, and it diverges to $ \infty$ in the high-disorder scenario.  To prevent underflow phenomena in the  low-disorder case ($a=1$), we have only considered values of $L$ less than $10$. We note that even though the value of $L$ is very small, we already observe very small values for the length.

In high-disorder cases $(a=9)$, we observe that the simulation and the theoretical value given in~\Cref{th1} are very close for $L<40$. However,  for larger $L$ the simulation values are much lower. This behavior is due to the fact that, for very large $L$, the function lives on increasingly higher frequencies (see~\cite{nostro} for more details). For computational reasons, our experiments were conducted using frequencies up to $1536$, and thus, for sufficiently large $L$, we inevitably lose information. This is confirmed by~\Cref{fig::variance}, which shows the percentage of variance explained by using the first $1536$ frequencies: as can be seen, the explained variance starts to decrease precisely at $L=40$.

\subsection*{Acknowledgements}This work was partially supported by the MUR Excellence Department Project MatMod@TOV awarded to the Department of Mathematics, University of Rome Tor Vergata, CUP E83C18000100006. We also acknowledge financial support from the MUR 2022 PRIN project GRAFIA, project code 202284Z9E4, the INdAM group GNAMPA and the PNRR CN1 High Performance Computing, Spoke 3.

\begin{figure}[!ht]
\centering
\subfloat[Low-disorder: $a=1$]{
	\includegraphics[width=0.45\linewidth]{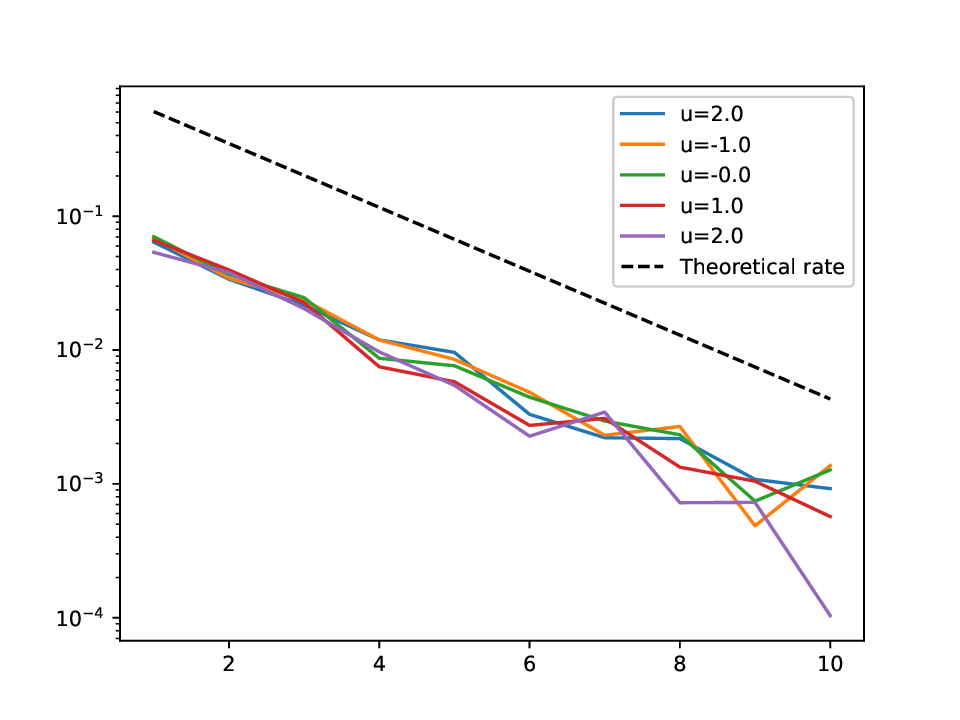}}
\subfloat[Sparse: $a= 1 +\sqrt 2$]{
	\includegraphics[width=0.45\linewidth]{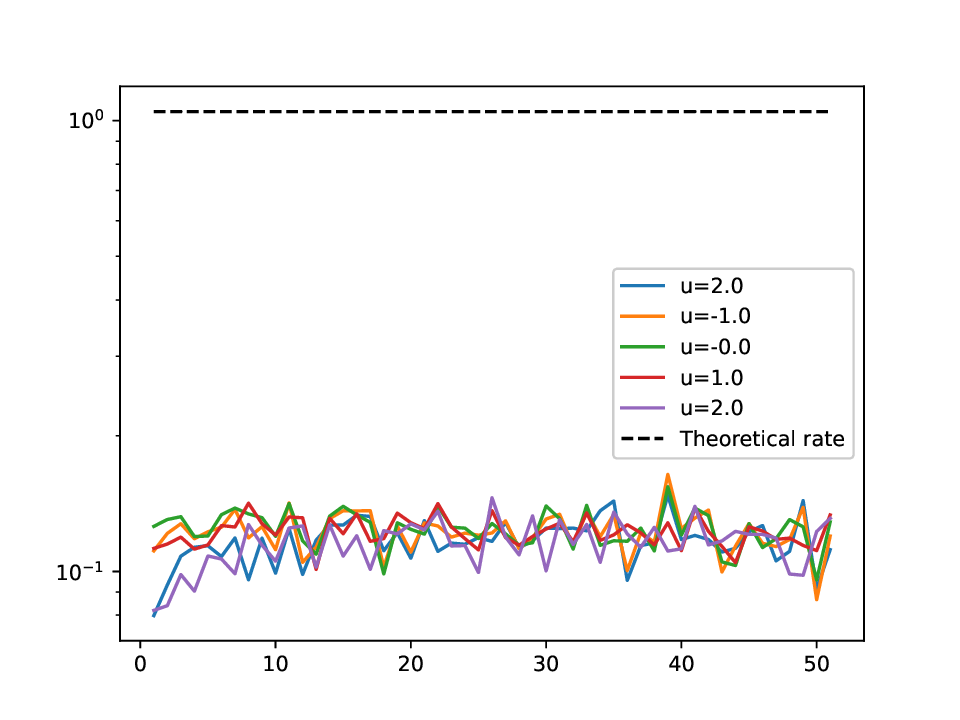}
}\\
\subfloat[High-disorder: $a=9$]{
	\includegraphics[width=0.45\linewidth]{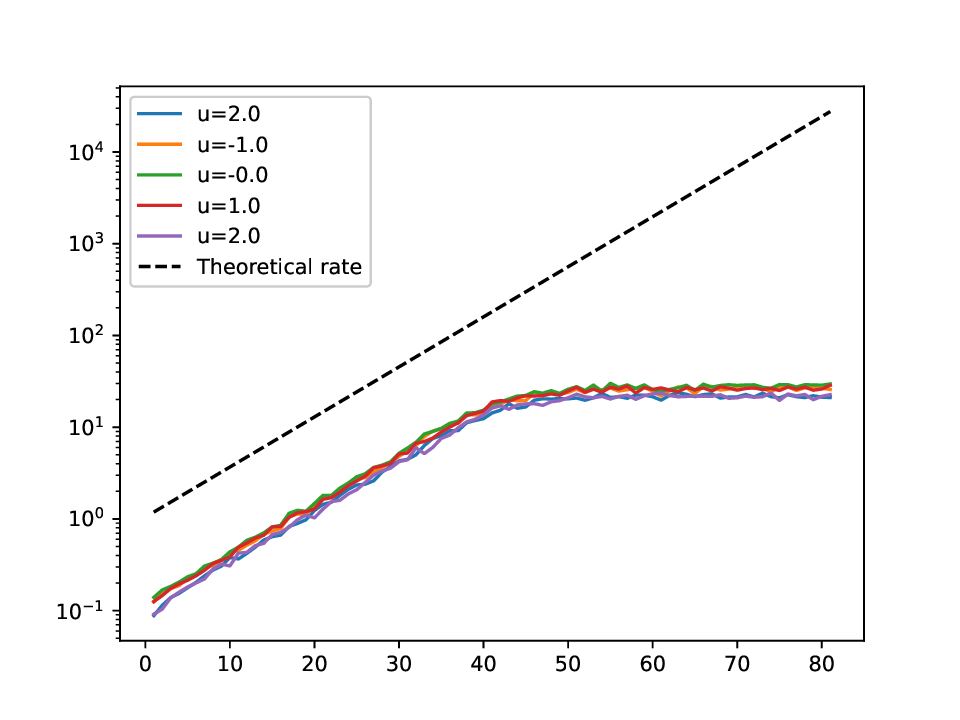}
}
\caption{Length of the  boundary of the excursion set for  Gaussian activation functions with different values of $a$ and  threshold $u$, as a function of the depth $L$. The dashed lines represent the theoretical rate of growth (i.e. $\kappa'(1)^{L/2}$). The resolution of the
	maps is $0.11$ deg. }
\label{fig::nodal_length}
\end{figure}
\begin{figure}[!h]
\centering
\includegraphics[width=0.45\linewidth]{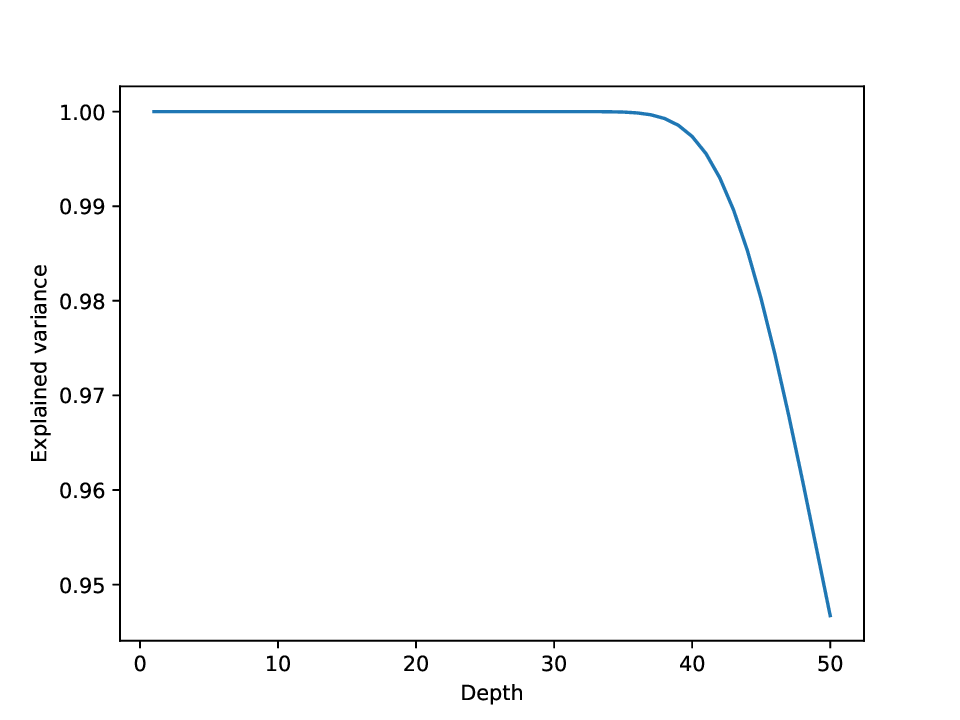}
\caption{Percentage of the variance of $T_L$ explained by  the first $1536$ frequencies in the case $a=9$. These values are obtained  using the Gauss-Legendre quadrature to compute the angular power spectrum.}
\label{fig::variance}
\end{figure}

\newpage

\bibliographystyle{plain}
\bibliography{ref}
\addcontentsline{toc}{section}{References}
\newpage
\appendix

\section{The relationship between spectral index and CRI }
In this appendix, we explore the relationship between the spectral index (cfr.~\Cref{spectral})  and covariance regularity (cfr.~\Cref{CRI}). Since the covariance function can be expressed in terms of Gegenbauer polynomials~\eqref{cova}, we first recall their definition and prove an integral representation.
\subsection{Gegenbauer polynomials}\label{gege}
Let $\alpha, \beta > -1$ be real numbers. The family of Jacobi polynomials with parameters $(\alpha, \beta)$ is the unique sequence $(P_\ell^{(\alpha,\beta)})_{\ell\in \mathbb{N}}$ satisfying the following properties (see~\cite{Szegoe} for more details):
\begin{itemize}
\item for each $\ell$, $P_\ell^{(\alpha,\beta)}$ is a polynomial of degree $\ell$,
\item if $\ell \neq \ell'$, then
$$ \int_{-1}^1 P_\ell^{(\alpha,\beta)} (x) P_{\ell'}^{(\alpha,\beta)}(x) (1-x)^\alpha (1+x)^\beta \, dx = 0 ,$$
\item if $\Gamma$ denotes the Euler Gamma function, then
$$ P_\ell^{(\alpha,\beta)}(1) = \frac{\Gamma(\ell+\alpha+1)}{\ell! \Gamma(\alpha+1)} .
$$
\end{itemize}

Let $d \geq 2$ be an integer. The Gegenbauer polynomials associated with the sphere $\S^d$ are the sequence of polynomials $(G_{\ell,d})_{\ell \in \mathbb{N}}$, where
$$ G_{\ell,d} (x) =   \frac{ \Gamma(d/2) \ell !}{\Gamma(\ell + d/2)} P_\ell^{(d/2-1 ,d/2-1)}(x) \ .
$$
In particular, it is easy to prove that $G_{\ell,d}(1) = 1$;  for $d=2$ this expression  corresponds  to the well-known Legendre polynomials.

The first step of our proof is to give an integral representation for the Gegenbauer polynomials. To establish this result, we will use the following formula (\cite{Szegoe}, eq. (4.10.3)), which holds for Jacobi polynomials with $\alpha = \beta = \lambda - 1/2$:
\begin{equation*}
\label{jac}
P_\ell^{(\lambda-1/2, \lambda-1/2)} (x) = 2^{1-2\lambda}
\frac{\Gamma(2\lambda) \Gamma(\ell+\lambda + 1/2)}{\Gamma(\lambda)^2 \Gamma(\lambda + 1/2)  \ell !} \int_0^\pi
\left(  x + (x^2-1)\cos \varphi \right)^\ell \sin(\varphi)^{2\lambda-1} \, \di \varphi .
\end{equation*}

The result below has some independent interest and may be known already, but we failed to locate any references, and hence we give a full proof.

\begin{lemma}[Integral representation of Gegenbauer polynomials]
\label{int_Gege} Let $\theta \in (-\pi/2, \pi/2)$. Then
$$ G_{\ell,d}(\cos \theta) =  \frac{2^{(3-d)/2} (d-2)!}{\Gamma\left(\frac{d-1}{2}\right)^2 |\sin(\theta)^{d-2}|}  \int_{0}^{\theta} \cos \left( \left(n + \frac{d-1}{2}\right)\psi \right) \left( 2\cos \psi - 2
\cos \theta \right)^{(d-3)/2} \, \di \psi .
$$
\end{lemma}
\begin{proof} Using the integral representation of Jacobi polynomial with $\lambda = (d-1)/2 $ and the definition of the Gegenbauer polynomials, we get
$$ G_{\ell,d}(x) = C_d \int_0^\pi \left( x + \sqrt{x^2-1}\cos\varphi \right)^\ell \sin^{d-2}(\varphi) \, \di\varphi ,$$
where
$$C_d = \frac{2^{2-d}(d-2)!}{\Gamma\left(\frac{d-1}{2}\right)^2 }  .$$
Substituting $h = x + \sqrt{x^2-1} \cos\varphi\in \mathbb C$, a simple calculation shows that
$$ h^2 - 2hx +1 = -(x^2-1)\sin^2\varphi .$$
Since $|x| < 1$ and $\sin \varphi \geq 0$ for $\varphi \in [0, \pi]$, we have
$$ \sin \varphi =\sqrt{\frac{h^2 - 2hx +1}{1-x^2}}.$$
Thus, with a change of variables, we obtain
$$ \frac{i\di h}{\sqrt{1-x^2}} = \sin \varphi \di \varphi \ .$$
Evaluating the integral yields
$$ G_{\ell,d}(x) = C_d \frac{-i}{(1-x^2)^{(d-2)/2}}\int_{x-\sqrt{x^2-1}}^{x+\sqrt{x^2-1}} h^\ell \left( h^2 - 2hx +1 \right)^{(d-3)/2} \, \di h .$$
Now, let $x = \cos \theta$ with $|\theta| < \pi/2$, then
$$ G_{\ell,d}(\cos\theta) =  C_d  \frac{-i}{|\sin^{d-2}(\theta)|} \int_{e^{-i\theta}}^{e^{i\theta}} h^\ell \left( h^2 - 2h\cos\theta +1 \right)^{(d-3)/2} \, dh .$$
Finally, the function is holomorphic, so we can integrate along any curve that starts and ends at the extremes. For example, we can integrate along the curve $h = e^{i\psi}$ for $\psi \in[-\theta, \theta]$, yielding
$$  G_{\ell,d}(\cos\theta) =  \frac{C_d}{|\sin^{d-2}(\theta)|} \int_{-\theta}^{\theta} e^{i\ell\psi} \left( e^{i2\psi} -2e^{i\psi}\cos\theta +1 \right)^{(d-3)/2} e^{i\psi} \, d\psi .$$
Note that
$$ e^{i2\psi}-2e^{i\psi} \cos \theta +1 = (2\cos\psi -2 \cos \theta) e^{i\psi }.$$
Therefore, we get
$$
G_{\ell,d}(\cos\theta) =  \frac{C_d}{\sin^{d-2}(\theta)} \int_{-\theta}^{\theta} e^{i(\ell+(d-1)/2)\psi} \left( 2\cos \psi - 2
\cos \theta \right)^{(d-3)/2} \, d\psi.
$$
Finally, since cosine is an even function and sine is odd, we obtain the claim.
\end{proof}

To prove the next two lemmas, we use the polylogarithmic function.  For any fixed $s>1$, we define
$$ \Li_s(z) = \sum_{k=1}^\infty \frac{z^k}{k^s},\qquad z\in \mathbb C,  \ |z|<1  $$
and we extend it holomorphically to the other values of $z$. For $|z|<1$, using  the equations (9.5) and (9.4) in~\cite{wood1992computation}, we have
\begin{align}\label{naturale}
\Li_s(z) = \frac{\ln(z)^{s-1}}{(s-1)!} \left(H_{s-1} - \ln(\ln(z^{-1}))\right) + \sum_{k=0
	\atop{k\neq n-1}}^\infty \frac{\zeta(n-k)}{k!} \ln(z)^k, s\in \N
	\end{align}
	\begin{align}
\label{not_naturale}
\Li_s(z) = \Gamma(1-s) (\ln z^{-1})^{s-1} + \sum_{k=0}^\infty \frac{\zeta(s-k)}{k!} (\ln z)^k, \; \;  \ &s\not\in \N \text{ and } |\ln z| <2\pi
\end{align}
where $H_n$ is the $n$-th harmonic number given by $H_0 = 0$ and
$$ H_n = \sum_{j=1}^n \frac{1}{j}, $$
while $\zeta$ is the Riemann zeta function.

Thanks to the introduction of the function $\Li_s$, the integral representation of the Gegenbauer polynomials, and two technical lemmas established in the next section, we can prove the following two lemmas. The first concerns non-integer values of $s$, while the second deals with integer values.

\begin{lemma}\label{lem10}
Let $s\not \in \N$.  Then for all integers $d\geq 2$, we have, as $\theta \to 0^+$,
\begin{align*} \sum_{\ell=1}^\infty \ell^{-s} G_{\ell;d}(\cos(\theta)) &= \begin{cases}
		\zeta(s) - K_1 \theta^{s-1}  + o(\theta^{s-1})
		& \text{ if } 1<s<3\ ,\\
		\zeta(s) - K_2 \theta^2  +K_3\theta^{s-1} +  o(\theta^{s-1} ) & \text{ if } 3<s<5 \ , \\
		\zeta(s) - K_2\theta^2  +K_4 \theta^{4} +  o(\theta^{4} ) & \text{ if } s\geq 5 \ ,
	\end{cases}
\end{align*}
where $\zeta$ is the Riemann zeta function and $K_1$ and $K_2$ are positive constants depending only on $s$ and $d$.
\end{lemma}
\begin{proof}
Let
$$ S_d = \frac{2^{(3-d)/2} (d-2)!}{\Gamma\left(\frac{d-1}{2}\right)^2}\ ,  $$
using the integral representation of Gegenbauer polynomials (see ~\Cref{int_Gege}) we have
\begin{equation}\label{e1}
	\begin{aligned}
		J &= \sum_{\ell=1}^\infty \ell^{-s} G_{\ell;d}(\cos \theta )=\frac{S_d}{|\sin^{d-2}(\theta)|}
		\sum_{\ell=1}^\infty \ell^{-s} \int_0^\theta\frac{ \cos\tonde{\tonde{\ell+ \frac{d-1} 2}\psi} }{|\cos\psi - \cos \theta|^{(3-d)/2}} \di \psi \\
		&=  \frac{S_d}{|\sin^{d-2}(\theta)|}\int_0^\theta \frac{1 }{|\cos\psi - \cos \theta|^{(3-d)/2}} \sum_{\ell=1}^\infty \ell^{-s} \cos\tonde{\tonde{ \ell + \frac{d-1}2}\psi}   \\
		&=\frac{S_d}{|\sin^{d-2}(\theta)|} \int_0^\theta \frac{1 }{|\cos\psi - \cos \theta|^{(3-d)/2}} \Re\tonde{\sum_{\ell=1}^\infty \ell^{-s} e^{ i \tonde{\ell + \frac{d-1}2}\psi}} \di \psi\\
		&= \frac{S_d}{|\sin^{d-2}(\theta)|}\int_0^\theta \frac{1 }{|\cos\psi - \cos \theta|^{(3-d)/2}} \Re\tonde{\Li_s(e^{i\psi}) e^{i\frac{d-1} 2 \psi}}\di \psi \ .
	\end{aligned}
\end{equation}
	%
	Since
	$$\ln (e^{i\psi} )^{k}= (i\psi)^{k} =\psi^{k} i^k = \begin{cases}
		(-1)^{k/2}\psi^k & \text{ if $k$ even}\\
		i(-1)^{(k-1)/2} \psi^k &\text{ if $k$ odd}
	\end{cases}
	$$
	and
	$$ \ln\tonde{  e^{-i \psi}}^{s-1} = \psi^{s-1} e^{-\pi/2(s-1)} = \psi^{s-1} \tonde{ \cos\tonde{\frac{\pi(s-1)} 2}  - i \sin \tonde{\frac{\pi(s-1)} 2 }}\ . $$
	Let $\Re(z)$ denote the real part of $z\in \mathbb C$; 	using~\eqref{not_naturale}, we obtain:
	\begin{equation*}
		\begin{aligned}
			\Re\tonde{Li_s(e^{i\psi}) e^{i\frac {d-1} 2 \psi}}  =& \cos\tonde{ \frac{d-1} 2 \psi}\Gamma(1-s) \psi^{s-1}   \cos\tonde{ \frac{s-1} 2 \pi} \\
			&				+\sin\tonde{ \frac{d-1} 2 \psi}\Gamma(1-s) \psi^{s-1}   \sin\tonde{ \frac{s-1} 2 \pi}\\
			&+ \cos\tonde{ \frac{d-1} 2 \psi}\sum_{k=0\atop{k\text{ even }}}^\infty(-1)^{k/2}   \frac{\zeta(s-k)}{k!} \psi^k \\
			& -\sin\tonde{ \frac{d-1} 2 \psi} \sum_{k=0\atop{k\text{ odd }}}^\infty(-1)^{(k-1)/2}   \frac{\zeta(s-k)}{k!} \psi^k  \ .
		\end{aligned}
	\end{equation*}
	Thus, we can rewrite equation~\eqref{e1} as:
	\begin{align*}J=&
		\Gamma(1-s)\cos\tonde{\frac {s-1} 2 \pi}
		\frac{S_d}{\sin^{d-2}(\theta)} \int_0^\theta \frac{\sin\tonde{\frac{d-1} 2\psi} \psi^{s-1} }{|2\cos\psi - \cos \theta|^{(3-d)/2}}  \di \psi \\
		& +  \sum_{k=0 \atop{k \text{ odd}}}  \frac{(-1)^{(k+1)/2}\zeta(s-k)}{k!}  \frac{S_d}{\sin^{d-2}(\theta)} \int_0^\theta \frac{\sin\tonde{\frac{d-1} 2\psi} \psi^{k} }{|2\cos\psi - \cos \theta|^{(3-d)/2}}  \di \psi
		\\
		&+\Gamma(1-s)\sin\tonde{\frac \pi 2 (s-1)}  \frac{S_d}{|\sin^{d-2}(\theta)|}
		\int_0^\theta \frac{\cos\tonde{ \frac {d-1} 2 \psi} \psi^{s-1}}{|\cos\psi - \cos \theta|^{(3-d)/2}}\di \psi \\
		& + \sum_{k=0 \atop{k \text{ even}}}  \frac{(-1)^{k/2}\zeta(s-k)}{k!} \frac{S_d}{|\sin^{d-2}(\theta)|}
		\int_0^\theta \frac{\cos\tonde{ \frac {d-1} 2 \psi} \psi^{k}}{|\cos\psi - \cos \theta|^{(3-d)/2}} \di \psi  \ .
	\end{align*}
	Using~\Cref{sin-cosa},  the previous equation becomes, as $\theta \to 0^+$,
	\begin{align*}
		&\frac{J \sin^{d-2}(\theta)}{S_d \sin^{d-2}(\theta/2)} \\
		&= \Gamma(1-s) \sin \tonde{\frac {s-1 }2  \pi} Q_{0,s-1,d}\sin^{s-1}(\theta/2)
		\\
		&\quad  + \Gamma(1-s) \cos \tonde{\frac {s-1 }2 \pi } Q_{1,s-1,d} \sin^{s}(\theta/2) \\
		&  \quad+\zeta(s) Q_{0,0,d}  +\sin^2(\theta/2) \tonde{  Q_{2,0,d} \zeta(s)  -\frac{\zeta(s-2)}{2} Q_{0,2,d}-\zeta(s-1) Q_{1,1,d}}\\
		& \quad+ \sin^4(\theta/2) \tonde{ Q_{4,0,d}\zeta(s)  - \zeta(s-1) Q_{3,1,d}   - \frac{\zeta(s-2)} 2 Q_{2,2,d} + \frac{\zeta(s-3)}{3!} Q_{1,3,d}+ \frac{\zeta(s-4)}{4!} Q_{0,4,d}}\\
		&\quad+ O \tonde{ \sin^{6}(\theta/2)} + O \tonde{ \sin^{s+1}(\theta/2) }
	\end{align*}
	where $Q_{a,b,d}$ are fixed constants given in~\eqref{Q_def}.
	Let us divide the proof in three different cases
	\begin{itemize}
		\item Let $1<s<3$, using the previous identity and since \begin{equation*}
			\label{st}\sin^{d-2}(\theta)  =2^{d-2} \sin^{d-2}(\theta/2)(1+o(1))\ ,
		\end{equation*}
		we obtain
		$$ J=A_d\zeta(s) - L_1
		\sin^{s-1} (\theta/2)+O \tonde{\sin^{\min(s, 2)}(\theta/2)}  \ .  $$
		where
		\begin{equation*}A_d =  2^{2-d}S_dQ_{0,0,d}  2^{(1-d)/2}
		\end{equation*}
		and
		$$ L_1= -2^{2-d}\Gamma(1-s) \sin\tonde{\frac {s-1} 2 \pi} S_dQ_{0,s-1,d} >0 \ , $$
		indeed, $\sin((s-1)/2 \pi ) >0$,  $\Gamma(1-s)$ and $Q_{0,s-1,d}>0$ (see~\Cref{sin-cosa} below).
		\item Let $3<s<5$. A similar computation shows that
		$$ J = A_d\zeta(s)-L_{2}\sin^2(\theta/2) +L_{3} \sin^{s-1}(\theta/2)+ O \tonde{\sin^{\min(s, 4)}(\theta/2)}$$
		where
		\begin{equation*} L_{2} = -2^{2-d}S_d \tonde{   \zeta(s) Q_{2,0,d} - \frac{\zeta(s-2)}{2} Q_{0,2,d}-\zeta(s-1) Q_{1,1,d}} >0
		\end{equation*}
		($Q_{2,0,d}<0$ and $Q_{0,2,d}, Q_{1,1,d}>0$)
		and
		$$L_{3} = \Gamma(1-s) \sin\tonde{\frac{s-1} 2 \pi} Q_{0,s-1,d} \ . $$
		\item Let $s>5$, then
		$$ J =  A_d\zeta(s)-L_{2}\sin^2(\theta/2) +L_{4}\sin^4(\theta/2)+ O \tonde{\sin^{6}(\theta/2)}
		$$
		where
		$$ L_4 = 2^{2-d} S_d\tonde{Q_{4,0,d}\zeta(s)  - \zeta(s-1) Q_{3,1,d}   - \frac{\zeta(s-2)} 2 Q_{2,2,d} + \frac{\zeta(s-3)}{3!} Q_{1,3,d}+ \frac{\zeta(s-4)}{4!} Q_{0,4,d} } \ . $$
	\end{itemize}
	If we prove that $A_d = 1$, then the claim holds. Using the definition of $Q_{0,0,d}$~\eqref{Q_def}, we have
	\begin{align*}
		A_d &= 2^{2-d} 2^{(3-d)/2} \frac{(d-2)!}{\Gamma\tonde{\frac{d-1}2}^2} 2^{(2-3)/2}  B \tonde{ \frac{1}2 , \frac{d-1} 2 }=
		\frac{	2^{2-d}(d-2)! \sqrt \pi }{\Gamma\tonde{\frac {d-1} 2 } \Gamma \tonde{\frac  d 2 }}=1
	\end{align*}
	where the last equality follows by the Legendre's duplication formula for the Gamma function (see \cite{abramowitz1968handbook}, eq. (6.1.18)). The claim follows using the expansion of $\sin(\theta/2)$.
\end{proof}
The following lemma extends the previous to integer values of $s$.

\begin{lemma} \label{lem:riemann} Let $s\geq 2 $ be a natural number. For all integer $d\geq 2$, we have, as $\theta \to 0^+$,
	$$\sum_{\ell=1}^\infty \ell^{-s} G_{\ell;d}(\cos(\theta)) = \begin{cases}
		\zeta(2) - K_5 \theta + K_6 \theta^2 \ln(\theta)+o(\theta^2\ln(\theta)) & \text{ if } s = 2\ , \\
		\zeta(3) - K_7\theta^2  |\ln \theta| + o(\theta^2\ln \theta)   & \text{ if } s =3 \  , \\
		\zeta(4) - K_2 \theta^2 + K_8 \theta^3 + o(\theta^3) & \text{  if } s = 4\ ,  \\
		\zeta(5) -K_2\theta^2 + K_9 \theta^4 \ln(\theta) +o(\theta^4 \ln\theta) & \text{ if } s = 5 \ , \\
		\zeta(s)  - K_2 \theta^2 + K_4 \theta^4 + o(\theta^4) & \text{ if } s\geq6  \ .	\end{cases}
	$$
	where $K_2$ and $K_4$ are as in the previous lemma and $K_5,K_7$ are positive.
\end{lemma}
\begin{proof}Let $s$ be a natural number. From~\eqref{naturale} we have
	\begin{equation}\label{naturale2}
		\Li_s(e^{i\psi}) = \frac{ (i\psi)^{s-1}}{(s-1)!} \tonde{ H_{s-1} - \ln(\psi) + i \frac{\pi} 2 } + \sum_{k=0\atop{k \neq s-1}}\frac{\zeta(s-k)}{k!}(i\psi)^k \ .
	\end{equation}
	We split the proof in two cases.
	\begin{itemize}
		\item If $s-1$ is even, then from~\eqref{naturale2} we obtain
		$$ \Li_s(e^{i\psi}) = \frac{ (\psi)^{s-1} (-1)^{(s-1)/2}} {(s-1)!} \tonde{ H_{s-1} - \ln(\psi) + i \frac{\pi} 2 } + \sum_{k=0\atop{k \neq s-1}}\frac{\zeta(s-k)}{k!}(i\psi)^k  $$
		and also
		\begin{align*}	\Re (e^{i\frac{d-1}2 \psi} \Li_s(e^{i\psi})) =& \sum_{k=0}^\infty \tonde{ C_k \cos\tonde{\frac{d-1} 2 \psi} \psi^k - S_k \sin\tonde{\frac{d-1} 2 \psi} \psi^k }\\
			&-  \frac{(-1)^{(s-1)/2}}{(s-1)!} \cos\tonde{\frac{d-1} 2 \psi} \psi^{s-1} \ln(\psi)\\
			&-  \frac{\pi(-1)^{(s-1)/2}}{2(s-1)!} \sin\tonde{\frac{d-1} 2 \psi} \psi^{s-1}  \ .	\end{align*}
		where
		$$ k! C_k = \begin{cases}
			{(-1)^{k/2} } \zeta(s-k) & \text{ if } k \text{  even and } k \neq s-1\\
			(-1)^{(s-1)/2}	H_{s-1}  & \text { if } k = s-1 \text{ even} \\
			0 & \text{ otherwise } \ ,
		\end{cases}$$

		$$ k! S_k = \begin{cases}
			{(-1)^{(k-1)/2} }\zeta(s-k) & \text{ if } k \text{  odd  }\ , \\
			0 & \text{ otherwise } \ .
		\end{cases}$$
		Using~\Cref{sin-cosa} and~\Cref{sin-cos-log} we have,
		\begin{align*}
			\frac{J_s}{S_d 2^{2-d}} =& \sum_{k=0}^\infty C_k\sin^k(\theta/2) \tonde{ Q_{0,k;d} + Q_{2,k;d} \sin^{2}(\theta/2) +Q_{4,k;d}\sin^4(\theta/2)+ O\tonde{\sin^{6}(\theta)/2}}\\
			& -  \sum_{k=0}^\infty S_k\sin^k(\theta/2) \tonde{ Q_{1,k;d} \sin(\theta/2) + Q_{3,k;d} \sin^3(\theta/2)+O\tonde{\sin^{5}(\theta)/2}}\\
			& - \frac{(-1)^{(s-1)/2} }{(s-1)!} \tonde{Q_{0,s-1;d} \sin^{s-1}(\theta/2) \ln(\sin(\theta/2)) + O\tonde{ \sin^{s-1}(\theta/2)}}\\
			&-  \frac{\pi(-1)^{(s-1)/2}}{2(s-1)!} \tonde{Q_{1,s-1;d} \sin^{s}(\theta/2) + O\tonde{ \sin^{s+2}(\theta/2)}}\\
			=& C_0 Q_{0,0;d}  + \tonde{C_0 Q_{2,0;d} + C_2 Q_{0,2;d}-S_1 Q_{1,1;d}}\sin^2(\theta/2) \\
			& +   \tonde{ C_0 Q_{4,0;d} + C_2 Q_{2,2;d} + C_4 Q_{0,4;d} - S_1 Q_{3,1;d}-S_3 Q_{1,3;d}} \sin^4(\theta/2)\\
			&+ 			O\tonde{\sin^6(\theta/2)}- \frac{(-1)^{(s-1)/2} }{(s-1)!} Q_{0,s-1;d} \sin^{s-1}(\theta/2) \ln(\sin(\theta/2)) + O\tonde{ \sin^{s-1}(\theta/2)} \ .
		\end{align*}
		Hence, if $L_2$ and $L_4$  are as in the previous proof we have:
		\begin{align*}
			J_3 &=\zeta(s)  -\frac{Q_{0,2;d}}  2\sin^{2}(\theta/2) |\ln(\sin(\theta/2))| +o(\tonde{\sin^2(\theta/2)\ln(\sin(\theta/2))} \quad & s = 3 \ , \\
			J_5 &= \zeta(s)-L_2 \sin^2(\theta/2) + \frac{Q_{0,4;d}}{24} \sin^4(\theta/2)\ln(\sin(\theta/2)) +			o\tonde{\sin^4(\theta/2)\ln(\sin(\theta/2))}  & s= 5\ ,  \\
			J_s & = \zeta(s) - L_2 \sin^2(\theta/2) +  L_4 \sin^4(\theta/2) + o(\sin^4(\theta/2)&  s\geq 7
			\text{ odd } \ .
		\end{align*}
		\item Let, now, $s$ be even. From~\eqref{naturale2} we have
		\begin{align*}
			\Re\tonde{ e^{i \frac{d-1} 2 \psi} \Li_s(e^{i\psi})} = &\sum_{k=0}^\infty \tonde{C_k' \cos\tonde{\frac{d-1} 2 \psi} \psi^k -S_k' \sin\tonde{\frac {d-1}2  \psi^k}}\\
			&- \frac{\pi(-1)^{(s-2)/2}}{(s-1)!} \cos\tonde{\frac{d-1} 2 \psi} \psi^{s-1}\\
			&- \frac{(-1)^{(s-2)/2}}{(s-1)!} \sin\tonde{\frac{d-1} 2 \psi}\ln(\psi)
		\end{align*}
		where
		$$ k! C_k ' = \begin{cases}
			(-1)^{k/2} \zeta(s-k) & \text{  if } k \text{ even} \ , \\
			0 & \text{ otherwise}  \ ,
		\end{cases}
		$$
		and
		$$
		k S_k' = \begin{cases} (-1)^{(k-1)/2} \zeta(s-k) & \text{  if } k\neq s-1 \text{ odd} \ , \\
			(-1)^{(s-2)/2} H_{s-1} & \text{ if } k = s-1 \text{  odd}\ , \\
			0 & \text{ otherwise}  \ .
		\end{cases}$$
		Using~\Cref{sin-cosa} and~\Cref{sin-cos-log}, we obtain
		\begin{align*}\frac{J}{S_d 2^{2-d}}= &
			C_0 Q_{0,0,d} +  \tonde{C_0 Q_{2,0;d} + C_2 Q_{0,2;d} -S_1 Q_{1,1;d}} \sin^2(\theta/2)\\
			& + \tonde{ C_0 Q_{4,0;d} + C_2 Q_{2,2;d } C_4 Q_{0,0;d} -S_1 Q_{3,1;d} - S_3 Q_{1,3;d}} \sin^4(\theta/2)+ O\tonde{ \sin^6(\theta/2)}\\
			& - \frac{\pi (-1)^{(s-2)/2} }{(s-1)!} \sin^{s-1}(\theta/2)  Q_{0,s-1,d}  + O\tonde{\sin^{s+1}(\theta/2)} \\
			& - \frac{(-1)^{(s-2)/2}}{(s-1)!} Q_{1,s-1,d} \sin^s(\theta/2) \ln(\sin(\theta/2)) + O \tonde{\sin^{s}(\theta/2)} \ .
		\end{align*}
		Hence
		$$			J_2 = \zeta(s)  - \pi S_d 2^{2-d}Q_{0,s-1,d} \sin(\theta/2) + L_3 \sin(\theta/2)^2 \ln(\sin(\theta/2)) +  o\tonde{\sin^2(\theta/2)\ln(\sin(\theta/2))}\ , $$
		$$			J_4= A_d \zeta(s)  -L_2 \sin^2(\theta/2) -\frac{\pi}{6} \sin^3(\theta/2) +   o\tonde{ \sin^3(\theta/2)}$$
		and, for any $s\geq 6$ even,
		$$			J_s  = \zeta(s) -L_2 \sin^2(\theta/2) +L_4 \sin^4(\theta/2) +o(\sin^4(\theta/2)) \ . $$
	\end{itemize}
\end{proof}

\subsection{Proof of~\Cref{sufficient_condition}} \label{ciccio}

Using the expansion \eqref{cova} of the covariance function in terms of Gegenbauer polynomials, the definition of $n_{\ell,d}$ and the form of the angular power spectrum in \Cref{spectral}, we have
\begin{align*}1 - \kappa(\cos(\theta))
	&=
	\frac{1}{\omega_d}\sum_{\ell=1}^\infty q(\ell^{-1}) \ell^{-(\alpha+d)} \frac{2\ell+d-1}{\ell} \frac{(\ell+d-2)!}{(d-1)!(\ell-1)!} (1-G_{\ell,d}(\cos(\theta))) \\
	&=
	\frac{1}{\w_d(d-1)!} \sum_{\ell=1}^\infty q(\ell^{-1}) \ell^{-(\alpha+d)}
	(2\ell+d-1) \left(\prod_{i=1}^{d-2} (\ell+i)\right)  (1-G_{\ell,d}(\cos(\theta))) \ .
\end{align*}
Suppose that $q(\cdot)$ has degree $N$. Then, in the above expression, $ G_{\ell,d}(\cos(\theta)) $ multiplies monomials in $\ell^{-1}$ with degree no lower than $ \alpha+1 $ and no higher than $ \alpha+d+N $,
with coefficients depending only on the coefficients of $q(\cdot)$ and on $d$. Calling $b_s$ such coefficients, we can thus write
\begin{align}\label{ss}
	1- \kappa(\cos(\theta)) & =   \sum_{\ell=1}^\infty \sum_{s=\alpha+1}^{\alpha+d+N}b_s \ell^{-s} (1-G_{\ell,d}(\cos(\theta))) \ .
\end{align}
We note that, for $s>1$, we have
	$$ \sum_{\ell=1}^\infty \ell^{-s} = \zeta(s) < \infty \ .$$
Then, using \Cref{lem:riemann},
$$Q_s(\theta)=  \sum_{\ell=1}^\infty \ell^{-s} (1 - G_{\ell,d}(\cos(\theta))< \infty\ . $$
Since the second sum in~\eqref{ss} has a finite number of terms,
$$1- \kappa(\cos(\theta)) =   \sum_{s=\alpha+1}^{\alpha+d+N} b_s \sum_{\ell=1}^\infty \ell^{-s} (1-G_{\ell,d}(\cos(\theta))) =  \sum_{s=\alpha+1}^{\alpha+d+N} b_s Q_{s}(\theta) \ .$$
Now, if $\theta_{x,y}$ is the angle beetween $x$ and $y$, using the previous lemma we have for $\theta_{x,y}\to 0^+$ that
\begin{align*}
	& \E[|T(x) -T(y)|^2]  = D_1\theta_{x,y}^{\alpha} +o(\theta_{x,y}^{\alpha})\qquad & \alpha\in[0,2)\ , \\
	&\E[|T(x) -T(y)|^2]  = D_2\theta_{x,y}^{2}|\ln(\theta_{x,y})| +o(\theta_{x,y}^{2}|\ln(\theta_{x,y})|)\qquad & \alpha= 2\ , \\
	& \E[|T(x) -T(y)|^2]  = D_3\theta^2_{x,y} + D_4 \theta_{x,y}^{\alpha} + o(\theta_{x,y}^{\alpha}) \qquad&\alpha\in(2,4) \ , \\
	&\E[|T(x) -T(y)|^2]  = D_5\theta^2_{x,y}+D_6\theta^4_{x,y}\log(\theta_{x,y})+o (\theta_{x,y}^ 4\log(\theta_{x,y})) \qquad&\alpha= 4\ ,
	\\
	&\E[|T(x) -T(y)|^2]  = D_5\theta^2_{x,y}+D_7\theta^4_{x,y}+o (\theta_{x,y}^4) \qquad&\alpha> 4,
	\ .
\end{align*}
Now, by isotropy we obtain $$ \E[|T(x)-T(y)|^2] = 2\kappa(1) - 2\kappa(\langle x,y\rangle )$$
and if we set $\langle x,y\rangle= 1-t$ we have, as $t\to 0^+$,
\begin{align*}
	& \kappa(1-t) = \kappa(1) +B_1 \arccos(1-t)^\alpha +o(\arccos(1-t)^\alpha)
	\qquad & \alpha\in[0,2) \ , \\
		& \kappa(1-t) = \kappa(1) +B_2 \arccos(1-t)^2 \ln(\arccos(1-t)) \\
		& \textcolor{white}{+}+o(\arccos(1-t)^2\ln(\arccos(1-t)))
	\qquad & \alpha=2\ , \\
	&\kappa(1-t) = \kappa(1) +B_2 \arccos(1-t)^2 +B_3 \arccos(1-t)^\alpha +o(\arccos(1-t)^\alpha)\qquad&\alpha \in (2,4)  \ , \\
	&\kappa(1-t) = \kappa(1) +B_4 \arccos(1-t)^2  \\
	& \textcolor{white}{+}+B_5\arccos(1-t)^{4}\log(\arccos(1-t))
	+ o(\arccos(1-t)^4\log(\arccos(1-t)))\qquad&\alpha=4  \ , \\
	&\kappa(1-t) = \kappa(1) +B_4 \arccos(1-t)^2  +B_5\arccos(1-t)^{4}
	+ o(\arccos(1-t)^4)\qquad&\alpha>4  \ .
\end{align*}
and, using the expansion of $\arccos(1-t)$ we obtain the claimed results.
%

\subsection{Technical lemmas}
In this section, we will compute the expansion as $\theta \to 0^+$ of the function
\begin{equation*}
	\theta \to  \int_0^\theta \frac{f\tonde{ \frac{d-1} 2\psi} }{|\cos\psi - \cos \theta|^{(3-d)/2}} \psi^a  \ln\psi^q \di \psi
\end{equation*}
for varying $a, d \in \N$ and $q \in \{0,1\}$, where $f(x) = \sin(x)$ or $\cos(x)$. \\

The first lemma that we prove allows expressing $\cos(n\psi)$ and $\sin(n\psi)$ in terms of powers of $\cos(\psi/2)$ and $\sin(\psi/2)$.
\begin{lemma}
	Let $n$ be integer and let
	$$ f_c(x) = \cos(x), \qquad f_s(x) = \sin (x) \ ;  $$
	then, for $i\in \s{c,s}$ we have
	$$ f_i(n\psi)= \sum_{k=0}^n \sum_{q=0}^{n-k} A^{(i)}_{k,q;n} \sin^{2q+k}(\psi/ 2) \cos^k (\psi/ 2) $$
	and
	$$ f_i\tonde{ n \psi -\psi/ 2}
	= \sum_{k=0}^n \sum_{q=0}^{n-k} \sin^{2q+k}(\psi/ 2) \cos^k(\psi/ 2)\tonde{  A^{(i)}_{k,q,n} \cos(\psi/ 2)+A^{[i]}_{k,q;n}\sin (\psi/ 2)}$$
	where
	\begin{align*}&A_{k,q;n}^{(s)} = (-1)^{(k-1)/2+q}  \frac{2^k n!}{k! q!(n-k-q)!}
		\ ,&  k \text{ odd } \ , \\
		&A_{k,q;n}^{(s)} = 0\ , & k \text{ even } \ ,\\
		&A_{k,q;n}^{(c)} = (-1)^{k/2+q}  \frac{2^k n!}{k! q!(n-k-q)!}\ ,  &  k \text{ even } \ , \\
		&A_{k,q;n}^{(c)} = 0
		\ ,  & k \text{ odd } \, \\
		&A_{k,q;n}^{[c]} = A_{k,q;n}^{(s)}  \ , \\
		& A_{k,q;n}^{[s]} =- A_{k,q;n}^{(c)} \ .
	\end{align*}

\end{lemma}

\begin{proof}
	From Euler's formula, we have
	$$ \left( \cos(\psi) + i\sin(\psi)\right)^n = \cos(n\psi) + i \sin (n\psi), $$
	which gives
	$$ \sin(n\psi) = \Im \left( \left( \cos(\psi) + i\sin(\psi)\right)^n \right) \quad \text{and} \quad \cos(n\psi) = \Re \left( \left( \cos(\psi) + i\sin(\psi)\right)^n \right) \ . $$

	Expanding the expression for \( \left( \cos(\psi) + i\sin(\psi) \right)^n \) using the binomial theorem, we obtain:
	$$ \left( \cos(\psi) + i\sin(\psi) \right)^n = \sum_{k=0}^n \binom{n}{k} i^k \sin^k(\psi) \cos^{n-k}(\psi) \  . $$

	Thus, the sine term becomes:
	$$ \sin(n\psi) = \sum_{k=0 \atop{k \text{ odd}}} (-1)^{(k-1)/2} \binom{n}{k} \sin^k(\psi) \cos^{n-k}(\psi) \  . $$

	Using the trigonometric identities:
	$$ \sin(\psi) = 2 \sin\left(\psi/2\right)\cos\left(\psi/2 \right), $$
	$$ \cos(\psi) = 1 - \sin^2(\psi/2)\ ,  $$

	we rewrite the sine expansion as:
	\begin{align*}
		\sin(n\psi) &= \sum_{k=0 \atop{k \text{ odd}}} (-1)^{(k-1)/2} \binom{n}{k} 2^k \sin^k \left({\psi}/{2}\right) \cos^k \left({\psi}/{2}\right) \left( 1 - \sin^2\left({\psi}/{2}\right) \right)^{n-k} \\
		&= \sum_{k=0 \atop{k \text{ odd}}}^n \sum_{s = 0}^{n-k} \binom{n}{k} \binom{n-k}{s} (-1)^{(k-1+2s)/2} 2^k \sin^{k+2s} \left( {\psi}/{2} \right) \cos^k \left( {\psi}/{2} \right).
	\end{align*}

	A similar procedure gives the expansion for the cosine. Finally, using the addition formulae for sine and cosine, the desired result follows.
\end{proof}

\begin{lemma}\label{sin-cosa}Let $d\geq 2$ be an integer and let $a>0$ be a real number. As $\theta \to 0^+$ we have:
	\begin{align*}&\int_0^\theta \frac{\cos\tonde{ \frac{d-1}{2}\psi} \psi^a }{|\cos(\psi) -\cos(\theta)|^{(3-d)/2}} \di \psi
		\\
		&= \sin(\theta/2)^{a+d-2 }\tonde{Q_{0,a;d} + Q_{2,a;d} \sin^2(\theta/2) +Q_{4,a;d}\sin^4(\theta/2)+ O\tonde{\sin^6(\theta/2)}} \ , \\
		&	\int_0^\theta \frac{\sin\tonde{ \frac{d-1}{2}\psi} \psi^a }{|\sin(\psi) -\cos(\theta)|^{(3-d)/2}} \di \psi  \\
		&= \sin(\theta/2)^{a+d-2 }\tonde{  Q_{1,a;d}\sin(\theta/2)+ Q_{3,a;d}\sin^3(\theta/2) + O\tonde{\sin^{5}(\theta/2)}}
	\end{align*}
	where
	\begin{equation}\label{Q_def}
		\begin{aligned}
			Q_{0,a;d} &= 2^{a+(d-3)/2} \B\tonde{\frac{a+1} 2 , \frac{d-1} 2 }  \ , \\
			Q_{1,a;d} & = 	2^{a+(d-3)/2} \B\tonde{\frac{a+2} 2 , \frac{d-1} 2 } (d-1) \ , \\
			Q_{2,a;d}& = -2^{(d-5)/2+a}\B\tonde{\frac{a+3} 2 , \frac{d-1} 2 } \frac{3d^2-9d+3-a} 3\quad & \text{ $d$ even} \ , \\
			Q_{2,a;d} & = -2^{(d-5)/2+a}\B\tonde{\frac{a+3} 2 , \frac{d-1} 2 } \frac{3d^2-9d-a} 3\quad & \text{ $d$ odd} \ .
		\end{aligned}
	\end{equation}
\end{lemma}
\begin{proof}To simplify the notation, we denote $f_c(x) = \cos(x)$ and $f_s(x) = \sin(x)$, and let
	$$
	J_i = \int_0^\theta \frac{f_i\left( \frac{d-1}{2}\psi \right) \psi^a}{|\cos(\psi) - \cos(\theta)|^{(3-d)/2}} \, \mathrm{d}\psi, \qquad i \in \{s, c\} \ .
	$$
	We divide the proof into two parts:
	\begin{itemize}
		\item Suppose $d$ is odd, with $d = 2n + 1$. From the previous lemma and noting that, for $0 < \psi < \theta < \pi/2$, we have
		\begin{equation*}
			|\cos(\psi) - \cos(\theta)| = 2 \sin^2(\theta/2) \left(1- \frac{\sin^2(\psi/2)}{\sin^2(\theta/2)}  \right) \ ,
		\end{equation*}
		and therefore
		\begin{align*}
			J_i = 2^{n-1} \sin^{2n-2}(\theta/2) \sum_{k=0}^n \sum_{q=0}^{n-k} A_{k,q;n}^{(i)} \int_{0}^\theta \frac{\sin^{2q+k}(\psi/2)\cos(\psi/2)^k \psi^a}{\left(1- \frac{\sin^2(\psi/2)}{\sin^2(\theta/2)}  \right)^{1-n}} \, \mathrm{d}\psi  \ .
		\end{align*}
		A simple computation, show that, as  $\psi \to 0^+$,
		\begin{equation}
			\label{psia}
			\psi^a = 2^a\sin^a(\psi/2)\tonde{1 + \frac a 6 \sin^2(\psi/2) + \frac {( 22 +5a)a} {360} \sin^4(\psi/2) + O(\sin^6(\psi/2)}
		\end{equation}
		and
		\begin{align*}
			\cos(\psi/2)^{(k-1)/2} \psi^a =& 2^a \sin^a(\psi/2) \bigg[ 1 + \frac {a-3k+3} 6 \sin^2(\psi/2)   +D_{a,k} \sin^4(\psi/2) + O\tonde{\sin^6(\psi/2)}		\bigg]
		\end{align*}
		where $D_{a,k} = \frac{(22+5a)a}{360} - \frac{(k-1)(k-3)}{8}  - \frac{a(k-1)}{12}$. Hence,
		\begin{align*}
			J_i = & 2^{n-1+a} \sin^{2n-2}(\theta/2) \sum_{k=0}^{n} \sum_{q=0}^{n-k} A_{k,q;n}^{(i)} \left[ \int_0^\theta \frac{\sin^{2q+k+a}(\psi/2)}{\left(1- \frac{\sin^2(\psi/2)}{\sin^2(\theta/2)}  \right)^{1-n}} \cos(\psi/2) \, \mathrm{d}\psi  \right.\\
			&+ \frac{a - 3k + 3}{6} \int_0^\theta \frac{\sin^{2q+k+a+2}(\psi/2)}{\left(1- \frac{\sin^2(\psi/2)}{\sin^2(\theta/2)}  \right)^{1-n}} \cos(\psi/2) \, \mathrm{d}\psi  +D_{a,k}  \int_0^\theta \frac{\sin^{2q+k+a+4}(\psi/2)}{\left(1- \frac{\sin^2(\psi/2)}{\sin^2(\theta/2)} \right)^{1-n}} \cos(\psi/2) \, \mathrm{d}\psi \\
			&+ \left. O\left( \int_0^\theta \frac{\sin^{2q+k+a+6}(\psi/2)}{\left(1- \frac{\sin^2(\psi/2)}{\sin^2(\theta/2)}  \right)^{1-n}} \cos(\psi/2) \, \mathrm{d}\psi \right) \right] \ .
		\end{align*}
		Thus, with the change of variables $x = \frac{\sin(\psi/2)}{\sin(\theta/2)}$, we obtain:
		\begin{align*}
			J_i =& 2^{n+a} \sin^{2n-2}(\theta/2)\sum_{k=0}^{n}\sum_{q=0}^{n-k} A_{k,q;n}^{(i)} \left[ \sin^{2q+k+a+1}(\theta/2)\int_0^1  \frac{x^{2q+k+a}}{\tonde{ 1-x^2}^{1-n}}  \di x  \right.\\
			&+  \frac{a-3k+3}6\sin^{2q+k+a+3}(\theta/2)  \int_0^1  \frac{x^{2q+k+a+2}}{\tonde{ 1-x^2 }^{1-n}} \di x  +D_{a,k}   \sin^{2q+k+a+5}(\theta/2)  \int_0^1  \frac{x^{2q+k+a+4}}{\tonde{ 1-x^2}^{1-n}} \di x
			\\
			&\left.
			+ O \tonde{ \sin^{2q+k+a+6}(\theta/2)}
			\right] \\
			= &2^{n+a-1} \sin^{2n-1+a}(\theta/2)\sum_{k=0}^{n}\sum_{q=0}^{n-k} A_{k,q;n}^{(i)}   \sin^{2q+k}(\theta/2) \bigg[ \B\tonde{\frac{2q+k+a+1} 2 , n}\\
			&+ \B\tonde{\frac{2q+k+a+3} 2 , n} \frac{a-3k+3} 6   \sin^{2}(\theta/2)+D_{a,k}' \sin^{4}(\theta/2) + O \tonde{\sin^6(\theta/2)}\bigg] \\
		\end{align*}
		where we have set $$D_{a,k}'  =  \B\tonde{\frac{2q+k+a+5} 2 , n} D_{a,k} \ . $$
		Using the definition of $A_{k,q;n}^{(i)}$ and recalling that $d = 2n + 1$, it follows that
		\begin{align*}
			J_c = &  \sin^{a+d-2} (\theta/2) \bigg[ 2^{(d-3)/2+a} \B\tonde{\frac{a+1} 2 , \frac{d-1} 2 } \\
			& -2^{(d-5)/2+a} \sin^2(\theta/2)\B\tonde{\frac{a+3} 2 , \frac{d-1} 2 } \frac{3d^2-9d+3-a} 3  \\
			& + \sin^4(\theta/2) Q_{4,a;d} + O\tonde{\sin^6(\theta/2)}
		\end{align*}
		and
		\begin{align*} J_s = & \sin^{a+d-2}(\theta/2) \bigg[2^{(d-3)/2 +a } \B\tonde{\frac{a+2} 2 , \frac{d-1} 2 }  (d-1) \sin (\theta/2)\\
			&+ Q_{3,a;d}\sin^3(\theta/2)  + O\tonde{\sin^5(\theta/2)}
		\end{align*}
		for some constants $Q_{3,a;d}$ and $Q_{4,a;d}$.
That is, the thesis follows for odd dimensions.

\item Now suppose $d = 2n$ is even. Using the previous lemma and similar calculations as in the first item, we obtain
\begin{align*}J_i =& 2^{n+3/2} \sin^{2n-3}(\theta/2)\sum_{k=0}^n \sum_{q=0}^{n-k} A_{k,q;n}^{(i)} \int_{0}^\theta \frac{\sin^{2q+k}(\psi/2) \cos^{k+1}(\psi/2) \psi^a}{\tonde{ 1-\frac{\sin^2(\psi/2) }{\sin^2(\theta/2)}}^{3/2-n}} \di \psi \\
	&+2^{n+3/2} \sin^{2n-3}(\theta/2)\sum_{k=0}^n \sum_{q=0}^{n-k} A_{k,q;n}^{[i]} \int_{0}^\theta \frac{\sin^{2q+k+1}(\psi/2) \cos^{k}(\psi/2) \psi^a}{\tonde{ 1-\frac{\sin^2(\psi/2) }{\sin^2(\theta/2)}}^{3/2-n}} \di \psi  \\
	= &  \beth_i + \daleth_i
	\ .
\end{align*}
Now, following the proof of the previous point, we can show that
\begin{align*}
	\beth_c =&  \sin^{a+d-2}(\theta/2)  \bigg[2^{(d-3)/2+a}\B \tonde{ \frac{a+1} 2 , \frac{d-1} 2 }   \\
	& -2^{(d-5)/2+a}  \B\tonde{\frac{a+3} 2,\frac{d-1} 2} \frac{3d^2 -3d-a} 3 \sin^{2}(\theta/2)\\
	& + \widetilde{Q_{4,a;d}} \sin^4(\theta/2)   + O\tonde{\sin^{6}(\theta/2)} \bigg] \ ,
\end{align*}
and
\begin{align*}
	\beth_s= &   \sin^{a+d-2}(\theta/2) \bigg[ 2^{a+(d-3)/2} d \B \tonde{\frac{a+2} 2 , \frac {d-1} 2}\sin(\theta/2)  \\
	& + \widetilde{Q_{3,a;d}} \sin^3(\theta/2) +O\tonde{\sin^{5}(\theta/2)}\bigg]
\end{align*}
for some constant $\widetilde{Q_{3,a;d}}$ and $\widetilde{Q_{4,a;d}}$.
As far as the second term is concerned, we have

\begin{align*}
	\daleth_c =& \sin^{a+d-2} (\theta/2) \bigg[  2^{(d-3)/2+a}
	\B\tonde{ \frac{a+3} 2 , \frac {d-1} 2} d \sin^2(\theta/2)\\
	&+\widehat{Q_{4,a;d}}\sin^{4}(\theta/2)+ O\tonde{\sin^{6}(\theta/2)}\bigg]\\
\end{align*}
and
\begin{align*}
	\daleth_s = & \sin^{a+d-2}(\theta/2) \bigg[ -2^{(d-3)/2+a} \B \tonde{\frac{a+2} 2, \frac{d-1} 2 } \\
	& +\widehat{Q_{3,a;d}} \sin^3)\theta/2) + O\tonde{ \sin^{5}(\theta/2) } \bigg]
\end{align*}
for some constant $\widehat{Q_{3,a;d}}$ and $\widehat{Q_{4,a;d}}$.
Hence,
\begin{align*}
	J_c =
	&\sin^{a+d-2}(\theta/2)
	\bigg[2^{a+(d-3)/2}   \B \tonde{ \frac{a+1} 2 , \frac{d-1} 2 }  \\
	&- 2^{(d-5)/2 +a} \B\tonde{\frac{a+3} 2 , \frac{d-1} 2 } \frac{3d^2 -9d-a}3
	\sin^2(\theta/2)\\
	&+ Q_{4,a;d} \sin^4(\theta/2) + O(\sin^{6}(\theta/2))\bigg]
\end{align*}
and
\begin{align*}
	J_s = & \sin^{a+d-2}(\theta/2) \bigg[ 2^{a+(d-3)/2} \B\tonde{\frac{a+2} 2 , \frac{d-1} 2 } (d-1) \sin(\theta/2)\\
	&+ Q_{3,a;d}\sin^3(\theta/2) + O(\sin^5(\theta/2)) \bigg]
\end{align*}
where
$$ Q_{3,a;d} = \widehat{Q_{3,a;d}}+\widetilde{Q_{3,a;d}}$$
and
$$ Q_{4,a;d} = \widehat{Q_{4,a;d}}+\widetilde{Q_{4,a;d}} \ .$$

\end{itemize}
\end{proof}

\begin{lemma}
\label{sin-cos-log}Let $d\geq 2$ be an integer and let $a>0$ be a real number. As $\theta \to 0^+$ we have:
\begin{align*}\int_0^\theta \frac{\cos\tonde{ \frac{d-1}{2}\psi} \psi^a \ln(\psi) }{|\cos(\psi) -\cos(\theta)|^{(3-d)/2}} \di \psi  = &Q_{0,a;d}\sin(\theta/2)^{a+d-2} \ln (\sin(\theta/2)) +O\tonde{ \sin^{a+d-2}(\theta/2)} \ , \\
\int_0^\theta \frac{\sin\tonde{ \frac{d-1}{2}\psi}\psi^a\ln(\psi)}{|\sin(\psi) -\cos(\theta)|^{(3-d)/2}} \di \psi  = &
Q_{1,a;d}\sin(\theta/2)^{a+d-1}  \ln (\sin(\theta/2))+ O\tonde{ \sin^{a+d-1}(\theta/2)}
\end{align*}
where $Q_{0,a;d}$ and $Q_{1,a;d}$ are the same constants given in~\Cref{sin-cosa}.
\end{lemma}

\begin{proof}Applying the logarithm to~\eqref{psia}, with $a=1$, we obtain for $\psi \to 0^+$
$$ \ln(\psi) = \ln(2) + \ln(\sin(\psi/2)) + \ln \tonde{ 1+ O(\sin^2(\psi/2))}
= \ln(\sin(\psi/2)) + O(1)  \ .$$
Thus, if $f_c(x) = \cos(x)$ and $f_s(x) = \sin(x)$, for each $i \in \{s, c\}$, it holds that
\begin{align*} \int_0^\theta \frac{ f_i\tonde{\frac{d-1} 2 } \psi^a \ln(\psi)}{|\cos(\psi) - \cos(\theta)|^{(3-d)/2}} \di \psi  =& \ln(\sin(\theta/2))\int_0^\theta \frac{ f_i\tonde{\frac{d-1} 2 } \psi^a}{|\cos(\psi) - \cos(\theta)|^{(3-d)/2}} \\
&+ O \tonde{  \int_0^\theta \frac{ f_i\tonde{\frac{d-1} 2 } \psi^a \ln\tonde{
			\frac{\sin(\psi/2)}{\sin(\theta/2)}}}{|\cos(\psi) - \cos(\theta)|^{(3-d)/2}}
}
\ .
\end{align*}
Now, the first integral is the one studied in the previous lemma. Therefore, it is enough to show that the second term is $O(\sin^{a+d-2}(\theta/2))$ in the case $i=c$ and $O(\sin^{a+d-1}(\theta/2))$ in the case $i=s$. Let
$$R_i = \int_0^\theta \frac{f_i\tonde{\frac{d-1} 2 } \psi^a \ln\tonde{
	\frac{\sin(\psi/2)}{\sin(\theta/2)}}}{|\cos(\psi) - \cos(\theta)|^{(3-d)/2}} \di \psi \ .  $$
	\begin{itemize}\item We assume $d = 2n + 1$ is odd. Following the proof of the previous lemma, we have
\begin{align*} R_i =& 2^{n-1+a} \sin^{a+2n-1}(\theta/2) \sum_{k=0}^n \sum_{q=0}^{n-k} A_{k,q;n}^{(i)} \sin^{2q+k}(\theta/2) \left[\int_0^1 \frac{x^{a+2q+k}}{(1-x^2)^{3/2-n}}\ln(x) \di x \right.\\
	&+ \left. O \tonde{ \int_0^1 \frac{ x^{a+2q+k+2}}{(1-x^2)^{1-n}} \ln(x) }\di x  \right]
\end{align*}
The claim follows by using the definition of $A_{k,q;n}^{(i)}$.
\item If $d=2n$ is even, then
\begin{align*} R_i =& 2^{n-3/2+a} \sin^{a+2n-2}(\theta/2) \sum_{k=0}^n \sum_{q=0}^{n-k} A_{k,q;n}^{(i)} \sin^{2q+k}(\theta/2) \left[2\int_0^1 \frac{x^{a+2q+k}}{(1-x^2)^{1-n}}\ln(x) \di x \right.\\
	&+ \left. O \tonde{ \int_0^1 \frac{ x^{a+2q+k+2}}{(1-x^2)^{1-n}} \ln(x) }\di x  \right]\\
	& + 2^{n-3/2+a} \sin^{a+2n-1}(\theta/2) \sum_{k=0}^n \sum_{q=0}^{n-k} A_{k,q;n}^{[i]} \sin^{2q+k}(\theta/2) \left[2\int_0^1 \frac{x^{a+2q+k+1}}{(1-x^2)^{1-n}}\ln(x) \di x \right.\\
	&+ \left. O \tonde{ \int_0^1 \frac{ x^{a+2q+k+3}}{(1-x^2)^{1-n}} \ln(x) }\di x  \right]
\end{align*}
hence, using the definition of $A_{k,q;n}^{[i]}$, the claim follows..
\end{itemize}
\end{proof}
%
%
%
%
%

\section{Proof of Proposition~\ref{teo7}}
\label{a2}
\Cref{teo7} is a generalization to any arbitrary dimension of Theorem 2 in~\cite{marinucci-xiao}. To obtain the results in any dimension, one can follow the proof in~\cite{marinucci-xiao} using~\Cref{gen} above instead of their Proposition 7. To prove~\Cref{gen}, let us recall that every sequence of spherical harmonics satisfies the following addition formula (see~\cite{MR2934227}, Theorem 2.9)
\begin{equation}\label{add_formula}
\sum_{m=1}^{n_{\ell;d}} Y_{\ell,m;d}(x) Y_{\ell,m;d}(y) = \frac{n_{\ell;d}}{w_d} G_{\ell;d}(\langle x, y \rangle) \ .
\end{equation}

Furthermore, we recall that spherical harmonics are eigenfunctions of the Laplace-Beltrami operator on the sphere, i.e.
$$ \Delta_{S^d} Y_{\ell,m;d}(x) = -\lambda_{\ell;d} Y_{\ell,m;d}(x)$$
where $\lambda_{\ell;d} = -\ell(\ell+d-1)$. We will choose a sequence of spherical harmonics
$(Y_{\ell m} \; | \; \ell\in \N, \ m = 1, \dotsc, n_{\ell,d})$  such that
\begin{equation*}
Y_{\ell,m;d}(N_d) = \sqrt{\frac{n_{\ell;d}}{w_d}} \delta_{m,1}
\end{equation*}
The next lemma ensures that such a sequence exists.
\begin{lemma}\label{polo-nord}
For any fixed dimension $d$, there always exists a triangular sequence of real spherical harmonics $(Y_{\ell,m;d})$ such that
$$
Y_{0, m;d} (N_d) = \sqrt{\frac{n_{\ell;d}}{\w_d}} \delta_{m,1} \ .
$$
\end{lemma}
\begin{proof}
We prove the claim by induction on $d$. For $d=1$,  we can simply take  $Y_{0,1;1}(x) = 1$ and, for every $\ell \geq 1$,
$$
Y_{\ell,1} (\theta) =\frac{\cos(\ell \theta)}{\sqrt \pi}, \qquad  Y_{\ell,2;1} (\theta)= \frac{\sin(\ell\theta)}{\sqrt \pi}
$$
which proves the claim. For the inductive step, note that following the construction in~\cite{MR2934227} Section 2.11, if $(Y_{\ell,m;d-1})$ is an orthonormal basis of real spherical harmonics in dimension $d-1$, then we can construct a basis of real spherical harmonics in $\S^d$: $(Y_{\ell,m;d})$, where
\begin{align*}&Y_{\ell,1}(x) = \tilde P_{\ell,d,0}(t(x)) Y_{0,1;d-1}(x_{d-1})\\
&Y_{\ell,i+n_{0,d-1}}(x) = \tilde P_{\ell,d,1}(t(x)) Y_{1,i;d-1}(x_{d-1})\quad & i = 1, \cdots, n_{1, d-1}\\
&Y_{\ell,i+ n_{0,d-1}+n_{1,d-1}}(x) = \tilde P_{\ell,d,2}(t(x)) Y_{2,i;d-1}(x_{d-1})\quad & i = 1, \cdots, n_{2, d-1}\\
& \qquad \qquad \qquad \qquad\qquad\vdots \\
&Y_{\ell,i+ \sum_{a = 0}^{k-1} n_{a,d-1}}(x) = \tilde P_{\ell,d,k}(t(x)) Y_{k,i;d-1}(x_{d-1})\quad & i = 1, \cdots, n_{k, d-1}\\
& \qquad \qquad\qquad\qquad \qquad \vdots \\
&Y_{\ell,i+ \sum_{a = 0}^{\ell-1} n_{a,d-1}}(x) = \tilde P_{\ell,d,\ell}(t(x)) Y_{\ell,i;d-1}(x_{d-1})\quad & i = 1, \cdots, n_{\ell, d-1}\\
\end{align*}
and if $x \in \S^d$, then we define $t(x)$ and $x_{d-1}$ as the unique pair such that $x = \begin{pmatrix} t(x)\\  \sqrt{1-t(x)^2} x_{d-1} \end{pmatrix}$ and
$$
\tilde P_{\ell,d,j}(t) = \frac{ (\ell+d-2)!}{\ell! \Gamma\tonde{\frac {d} 2 }} \sqrt{ \frac{ (2\ell+d-1)(\ell-j)! }{2^{d-1}(\ell + d+ j-2)!}} (1-t^2)^{j/2} \frac{ \di^j }{\di t^j} G_{\ell;d}(t) \ .
$$
Note that for every $j \neq 0$, we have $\tilde P_{\ell,d,j}(1) = 0$. Furthermore, $t(N_d) = 1$, so for every $i > 1$, we have
$$ Y_{\ell,i;d}(N_d) = 0 \ . $$
Using the addition formula (cfr.~\eqref{add_formula}), we obtain
\begin{equation*}
\sum_{m=1}^{n_{\ell,d}} Y_{\ell,m}(N_d) {Y_{\ell,m}(N_d)} = \frac{n_{\ell,d}}{w_d} G_{\ell;d}(1) \ .
\end{equation*}
The claim follows by noting that the normalized Gegenbauer polynomials satisfy $G_{\ell;d}(1) = 1$.
\end{proof}

The next lemma is instrumental to the proof of~\Cref{gen}.\\
\begin{lemma}
\label{bb}
Let $T$ be an isotropic random field with $\beta<1$ as CRI. For all $0<\varepsilon<1$, there exists a triangular sequence $(\kappa_{\ell,m;d}(\varepsilon)\; | \; \ell\in \N, \ m = 1, \dotsc, n_{\ell,d})$ such that
\begin{align}\label{prop1} \sum_{\ell=0}^\infty \sum_{m=1}^{n_{\ell,d}} \frac{\kappa_{\ell,m;d}(\varepsilon)^2}{C_\ell}  \leq \frac{1}{c_1 \varepsilon^{2\beta+2d}} \ , \\
\label{prop2} \sum_{\ell=0}^\infty \sum_{m=1}^{n_{\ell,d}} \kappa_{\ell,m;d}(\varepsilon)Y_{\ell,m;d}(N_d) \geq \frac{c_2} {\varepsilon^d}\ , \\
\label{prop3} \sum_{\ell=0}^\infty \sum_{m=1}^{n_{\ell,d}} \kappa_{\ell,m;d}(\varepsilon) Y_{\ell,m;d}(x) =0 \qquad \text{ if } d_{\S^d}(x,N_d) > \varepsilon \ ,
\end{align}
where $c_1, c_2$ are positive constants.

\end{lemma}

\begin{proof} First we recall that, using~\cite{bietti2021deep}, Theorem 1, there exists $c_0>0$ such that
$$ c_0^{-1} \ell^{-(d+2\beta)}\leq C_\ell  \leq c_0\ell^{-(d+2\beta)}$$
Let $d\geq 2$ be a fixed integer and
$$ p(s) = \max\s{0, 1-2|s|}   \ . $$
We define
$$\hat{G}= p \star \dotsb \star p$$
where the function $p$ appears $d$ times and $\star$ denotes the convolution, i.e.
$$ (f\star g )(t) = \int_{-\infty}^\infty f(\tau) g(t-\tau) \di \tau \ . $$
Exploiting the relationship between  convolution and Fourier transform, we have that $\hat G$ is piecewise smooth. Letting $G$ be the inverse transform of $\hat{G}$, since
$$ \frac{2}{\pi} \int_{\R} p(x) e^{i x u} \, \mathrm{d}x =  \frac{2}{\pi u^2}(1-\cos(u/2)) \ , $$
we have
$$ G(u) = \tonde{\frac{2}{\pi}}^d (1-\cos(u/2))^d u^{-2d} \ . $$
In particular,
$$ \|G\|_\infty < C_G \ , $$
as $G(u) \to 0$ for $u \to +\infty$ and $G(u) \to \frac{2^{d-1}}{\pi^d}$ for $u \to 0$.\\
We now prove that
$$ \kappa_{\ell,m;d} = G(\varepsilon\sqrt{\lambda_{\ell;d}}) Y_{\ell,m;d}(N_d)$$
satysfies the three claimed properties.
\begin{itemize}
\item[\eqref{prop1}]Using the previous lemma, we have
$$ \kappa_{\ell,m;d} (\varepsilon)^2 = \begin{cases} 0 & \text{ if } m \neq 1 \\
	\frac{n_{\ell;d}} {\w_d} & \text{ if }  m=1
\end{cases} \ . $$
We note that
$$
n_{\ell,d} \leq  \frac{2\ell+(d-1)\ell}{(d-1)!}  \prod_{i=1}^{d-2} \ell(i+1) \leq  \frac{\ell^{d-1}}{(d-2)!} (d-1)!(d-1+2) = \ell^{d-1} d(d-1) \ .
$$
	%
	Now,
	$$ G(\varepsilon\sqrt{\lambda_\ell}) = \int_{-\infty}^\infty \hat G(s) e^{-is\varepsilon \sqrt{\lambda_\ell}} \di s = \int_{-\infty}^\infty \hat G^{(r)}(s) e^{-is\varepsilon\sqrt{\lambda_\ell}} \frac{1}{(-i\varepsilon \sqrt{\lambda_\ell})^r }\di s $$
	when the last inequality follows  integrating by parts $r$ times.  Hence, for each $r$ smaller than the number of derivatives of $\hat{G}$, we have
	$$ |G(\varepsilon\sqrt{\lambda_\ell})|^2 \leq \frac{K_r}{(\ell \varepsilon)^r} \ . $$
	Thus, combining the previous bounds
	we get, for any $M$,
	\begin{align}\nonumber
		\sum_{\ell=0}^\infty \sum_{m=1}^{n_{\ell,d}} \frac{\kappa_{\ell,m;d}(\varepsilon)^2}{C_\ell}  & = \sum_{\ell=1}^\infty \frac{n_{\ell}}{C_\ell \w_d} G(\varepsilon\sqrt{\lambda_\ell})^2
		\\
		&	\leq \sum_{\ell=M}^\infty  d(d-1) \ell^{d-1} c_0\ell^{d+2\beta} \frac{K_r}{(\ell \varepsilon)^r} + \sum_{\ell=1}^M \frac{n_{\ell}}{C_\ell \w_d} G(\varepsilon\sqrt{\lambda_\ell})^2 \nonumber \\
		&= \frac{c_0 d(d-1)K_r}{\varepsilon^{2\beta+2d}} \sum_{\ell=M}^\infty
		(\varepsilon\ell)^{2\beta+2d-1-r} \varepsilon + \sum_{\ell=1}^M \frac{n_{\ell}}{C_\ell \w_d} G(\varepsilon\sqrt{\lambda_\ell})^2 \ . \label{e2}
	\end{align}
	Now, regarding the first sum in~\eqref{e2}, if $r \geq 2\beta+2d$, for   $M = \left\lfloor 1/\varepsilon \right \rfloor$, we get
	$$ \sum_{\ell=M}^\infty
	(\varepsilon\ell)^{2\beta+2d-1-r} \varepsilon \leq  \int_{M\varepsilon}^\infty x^{2\beta + 2d -1-r} \, \mathrm{d}x \leq \int_{1}^\infty x^{2\beta + 2d -1-r} \, \mathrm{d}x  \ . $$
	Meanwhile, for the second sum in~\eqref{e2}, we have, for $c = c_0d(d-1)$
	$$ \sum_{\ell=1}^M \frac{n_{\ell}}{C_\ell \w_d} G(\varepsilon\sqrt{\lambda_\ell})^2 \leq c\sum_{\ell=1}^M \ell^{2\beta+2d-1}  \leq  cMM^{2\beta+2d-1}  = cM^{2\beta+2d} \leq c \varepsilon^{-2(\beta+d)} $$
	so the claim follows.
	\item[\eqref{prop2}]Using the previous lemma, we have
	\begin{equation*}\label{delta0}\begin{aligned}
			&\sum_{\ell=1}^\infty\sum_{m=1}^{n_{\ell;d}} \kappa_{\ell,m;d}(\varepsilon) Y_{\ell,m;d}(N_d)= \sum_{\ell=1}^\infty G(\varepsilon\sqrt{ \ell(\ell+d-1)}) \sum_{m=1}^{n_{\ell}} Y_{\ell,m} (N_d) Y_{\ell,m} (N_d)\\
			& = \sum_{\ell=1}^\infty G(\varepsilon\sqrt{ \ell(\ell+d-1)}) \frac{n_{\ell;d}}{\w_d}
			\\
			&\geq\frac{1}{(d-1)!\w_d}  \sum_{\ell=1}^\infty G(\varepsilon\sqrt{\ell(\ell+d-1)}) \ell^{d-1} \\
			& =\frac{1}{\varepsilon^d (d-1)! \w_d} \sum_{\ell=1}^\infty G(x_\ell) x_\ell^{d-1}(x_{\ell+1}- x_{\ell}) \tonde{ \frac{ \ell}{\ell+d-1}}^{(d-1)/2} \frac{ \sqrt{(\ell+1)(\ell+d)} + \sqrt{\ell(\ell+d-1)}}{
				2\ell +d }
		\end{aligned}
	\end{equation*}
	where we have set  $x_\ell = \varepsilon\sqrt{\ell(\ell+d-1)}$ and the inequality follows by
	\begin{equation*} \label{min_nld}
		n_{\ell,d} \geq \frac{\ell}{(d-1)!} \prod_{i=1}^{d-2} \ell  = \frac{\ell^{d-1}}{(d-1)!}  \ .
	\end{equation*}
	Now, since, for $\ell \to + \infty$,
	$$ b_\ell= \tonde{ \frac{ \ell}{\ell+d-1}}^{(d-1)/2} \frac{ \sqrt{(\ell+1)(\ell+d)} + \sqrt{\ell(\ell+d-1)}}{
		2\ell +d } \to 1 $$
	there exists an $L$ such that
	$$ \sum_{\ell=L}^\infty G(x_\ell) x_\ell^{d-1}(x_{\ell+1}- x_{\ell}) b_\ell \geq \frac{1}{2} \sum_{\ell=L}^\infty G(x_\ell)x_{\ell}^{d-1}(x_{\ell+1}-x_\ell)  \ .$$
	Since $(x_\ell)_{\ell=L}^\infty$ forms a partition of $[\varepsilon\sqrt{L(L+d-1)}, \infty)$, we have
	$$ \sum_{\ell=1}^\infty\sum_{m=1}^{n_{\ell;d}} \kappa_{\ell,m;d}(\varepsilon) Y_{\ell,m;d}(N_d)\geq \frac{1}{\varepsilon^{d} 2(d-1)\w_d}\int_{\varepsilon\sqrt{L(L+d-1)}}^\infty
	G(u) u^{d-1} \di u$$
	and for $\varepsilon<1$ the claim follows, putting
	$$ c_2 =\frac{1}{2(d-1)\w_d} \int_{\sqrt{L(L+d-1)}}^\infty G(u) u^{d-1} \di u <\infty  \ .$$

	\item[\eqref{prop3}]
	We now define the operator $G(\varepsilon\sqrt{-\Delta_{\S^d}}):\, L^2(\S^d) \to L^2(\S^d)$ in the standard way: given a triangular sequence $(a_{\ell,m})$ such that
	\begin{equation}\label{fgen}
		f(x) = \sum_{\ell=0}^\infty \sum_{m=1}^{n_{\ell,d}} a_{\ell,m}Y_{\ell,m;d}(x) \ ,
	\end{equation}
	we  define
	$$
	G(\varepsilon\sqrt{-\Delta_{\S^d}})f(x) = \sum_{\ell=0}^\infty G(\varepsilon\sqrt{\lambda_\ell}) \sum_{m=1}^{n_{\ell,d}} a_{\ell,m}Y_{\ell,m;d}(x) \ .
	$$
	From Huygens' principle~\cite[Lemma 4.1]{geller2009continuous}, if $K_\varepsilon$ is the kernel of the previous operator, then
	$$
	K_\varepsilon(x,y) = 0 \quad \text{if} \quad d_{\S^d}(x,y) > \varepsilon \ .
	$$
	To conclude, it is sufficient to show that
	$$K_\varepsilon(N_d, \cdot) = \sum_{\ell=1}^\infty \sum_{m=1}^\infty \kappa_{\ell,m;d}(\varepsilon) Y_{\ell,m;d}(\cdot) \ , $$ which is equivalent to proving that for every $f \in L^2(\S^d)$, we have
	\begin{equation*}
		G(\varepsilon \sqrt{-\Delta_{\S^d}}) f(N_d) = \int_{\S^d} \sum_{\ell=1}^\infty \sum_{m=1}^\infty \kappa_{\ell,m;d}(\varepsilon) Y_{\ell,m;d}(x)f(x) \, \mathrm{d}x \ .
	\end{equation*}
	If $f$ is of the form in~\eqref{fgen}, then the previous integral becomes
	$$ C = \int_{\S^d}
	\left( \sum_{\ell=0}^\infty G(\varepsilon\sqrt{\lambda_{\ell;d}}) \sum_{m=1}^{n_{\ell;d}}Y_{\ell,m;d}(N_d) Y_{\ell,m;d}(x) \right)
	\left( \sum_{\ell'=0}^\infty \sum_{m'=1}^{n_{\ell';d}} a_{\ell',m'}Y_{\ell',m';d}(x) \right) \, \mathrm{d}x \ .
	$$
	Using the orthonormality of the spherical harmonics, we obtain
	$$
	C= \sum_{\ell=0}^\infty G(\varepsilon\sqrt{\lambda_\ell}) \sum_{m=1}^{n_{\ell;d}} a_{\ell,m} Y_{\ell,m;d}(N_d) \ .
	$$

\end{itemize}
\end{proof}

Let us prove our last result.

\begin{proposition}\label{gen}
Under the assumptions of~\Cref{teo7}, we have
$$ \sum_{\ell=1}^\infty \sum_{m=1}^{n_{\ell,d}} C_\ell
\tonde{ Y_{\ell,m;d}(N_d) - \sum_{j=1}^n \gamma_j Y_{\ell,m;d}(x_j)}^2 \geq c \varepsilon^{2\beta} \ . $$
\end{proposition}

\begin{proof}Let  $(\kappa_{\ell,m}(\varepsilon))_{\ell\in N,\;  m  =1, \dotsc, n_{\ell,d}}$  be as in the previous lemma. We define
$$ A= A(x_1, \dotsc, x_n) = \sum_{\ell=1}^\infty \sum_{m=1}^{n_{\ell,d}} {C_\ell}
\tonde{ Y_{\ell,m;d}(N_d) - \sum_{j=1}^n \gamma_j Y_{\ell,m;d}(x_j)}^2 \ , $$
and
\begin{align*}
I = I(x_1, \dotsc, x_n) = \sum_{\ell=1}^\infty \sum_{m=1}^{n_{\ell,d}}
\frac{\kappa_{\ell,m;d}(\varepsilon)}{\sqrt{C_\ell}}
\quadre{ \sqrt{C_\ell}
	\tonde{ Y_{\ell,m;d}(N_d) - \sum_{j=1}^n \gamma_j Y_{\ell,m;d}(x_j)}} \ .
	\end{align*}
	From the Cauchy-Schwarz inequality, we obtain
	\begin{align*}
I^2 & \leq A \sum_{\ell=1}^\infty \sum_{m=1}^{n_{\ell}}\frac{\kappa_{\ell,m;d}(\varepsilon)^2}{C_\ell}   \ .
\end{align*}
Hence, from~\eqref{prop1} we obtain
\begin{equation}
\label{aa}A \geq c_1 I^2 \varepsilon^{2d+ 2\beta} \ .
\end{equation}
On the other hand,
$$ I = \sum_{\ell=1}^\infty \sum_{m=1}^{n_{\ell,d}} \kappa_{\ell,m;d}(\varepsilon) Y_{\ell,m;d}(N_d) -\sum_{j=1}^n \gamma_j \sum_{\ell=1}^\infty \sum_{m=1}^{n_{\ell,d}}
\kappa_{\ell,m;d}(\varepsilon) Y_{\ell,m;d}(x_j) \ .
$$
Using~\eqref{prop3} the second summation is equal to zero and by~\eqref{prop2} we obtain
$$ I \geq \frac{c_2}{\varepsilon^d}  \ .$$
The claim follows combining the previous inequality with~\eqref{aa}.
\end{proof}

\section{ Lower bound of~(\ref{almost}): Hausdorff dimension of the excursion set}\label{App_HAU}
Let $\mathcal M^+$ be the set of all positive measures on $\R^{d+1}$.  For any $\alpha>0$, we define the following inner product on $\mathcal{M}^+$,
$$ \langle \mu, \nu \rangle_\alpha  = \iint_{(\R^{d+1})^2} \frac{1}{\| x- y\|^{d+1-\alpha}}  \di \mu(x) \di \nu(y) \ , $$
and set $\mathcal{ E}_\alpha^+ =\{ \mu \in \mathcal M^+ \; | \; \|\mu\|_\alpha<\infty\}$.
It is well-known that $(\mathcal E^+_\alpha, \|\cdot \|_\alpha)$ is a Hilbert space~\cite[Theorem 1.18]{landkof}. The following proposition implies the lower bound of~\eqref{almost} as explained in~\Cref{pp111}.

\begin{proposition} Let $T$ be a zero mean, unit-variance isotropic Gaussian random field on $\S^d$. If the CRI $\beta\in (0,1]$, then, for every $u\in \R$, there exists a random measure $\mu$ on $\S^d$ such that, for every $\gamma<d-\beta$,
	$$ \int_{(\S^d)^2} \frac{1}{\| x-y\|^{\gamma}} \di\mu(x) \di \mu(y) < \infty\qquad \text{a.s.}$$
	and such that $\mu(T^{-1}(u))>0$ with positive probability.
\end{proposition}
\begin{proof}
	For any $n\in \N$, let $\mu_n$ be the random measure on $\R^{d+1}$  such that,  for any Borel set  $A\subseteq \R^{d+1}$,
	$$ \mu_n(A) = \int_{A\cap \S^d} \phi_n(T(x)) \di x, \qquad \text{where} \quad  \phi_n(t) = \sqrt{{2n\pi}} e^{-\frac{n(t-u)^2} 2}\,.$$
	Let $n,m\in \N$ and $\alpha>0$. Then, setting $ \gamma = d+1-\alpha $ we have
	\begin{align*}
		& \E[ \langle \mu_n , \mu_m\rangle_\alpha] =  \E\bigg[ \int_{(\R^{d+1})^2} \frac{1}{\| x-y\|^{\gamma}} \di \mu_n(x) \di \mu_m(y)\bigg]\\
		&\qquad =  2\pi \sqrt{nm}\E\bigg[\int_{(\S^d)^2} \exp\Big( -\frac{1}{2}(n (u-T(x))^2 + m (u-T(y))^2)\Big) \| x-y\|^{-\gamma} \di x \di y\bigg]\\
		& \qquad  = 2\pi \sqrt{nm}\int_{(\S^d)^2} \daleth(x,y) \| x-y\|^{-\gamma} \di x \di y \ ,
	\end{align*}
	where the last identity follows from Fubini's theorem and
	$$\daleth(x,y) =\E\bigg[\exp\Big( -\frac{1}{2}(n (u-T(x))^2 + m (u-T(y))^2)\Big) \bigg]\ .$$
 Since $T$ is a zero-mean Gaussian random field  with covariance kernel $\kappa$, then putting $\rho = \kappa(\langle x, y\rangle)$ we have $$(u-T(x), u-T(y))\sim \mathcal N(\boldsymbol{u}, \Sigma),\qquad \boldsymbol{u} =\begin{pmatrix}
		u\\
		u
	\end{pmatrix}, \quad \Sigma = \begin{pmatrix} 1 & \rho\\
		\rho&1
	\end{pmatrix} \ . $$
	Hence, for $\boldsymbol{s} = (s_1,s_2)^\top \in \R^2$ we have
	\begin{align*}  \daleth(x,y) =&
		 \int_{\R^2}\frac{1}{2\pi \sqrt{\det \Sigma}} e^{ -\frac{1}{2}\big( ns_1^2 + ms_2^2+(\boldsymbol{s} - \boldsymbol{u})^\top \Sigma^{-1} (\boldsymbol{s} - \boldsymbol{u})\big)}  \di \boldsymbol{s} \ .
	\end{align*}
Using some algebraic manipulation, one can prove that
$$ ns_1^2 + ms_2^2+(\boldsymbol{s} - \boldsymbol{u})^\top \Sigma^{-1} (\boldsymbol{s} - \boldsymbol{u})  = (\boldsymbol{s} - \boldsymbol{\mu}_{nm})^\top \Sigma_{mn}^{-1} (\boldsymbol{s} - \boldsymbol{\mu}_{nm}) + h_{nm} $$
where
$$ \Sigma_{mn}^{-1} = \frac{1}{1-\rho^2}\begin{pmatrix}
 n(1-\rho^2) + 1 & - \rho\\
-\rho &  m(1-\rho^2)
 	\end{pmatrix}, \qquad  \boldsymbol{\mu_{nm}} = \begin{pmatrix} \mu_1^{nm}\\ \mu_2^{nm}
\end{pmatrix}$$
$$h_{nm} =u^2\frac{2mn(1-\rho)+m+n}{mn(1-\rho^2)+m+n}  $$
with
$$
\mu_1^{nm} = u \frac{mn \rho^2 - mn - m \rho -n}{mn\rho^2-mn-mn -m -n -1}, \qquad \mu_2^{nm} = u \frac{mn \rho^2 - mn - n \rho -m}{mn\rho^2 -mn -m -n -1}\ .$$
Hence
	\begin{align*} & \daleth(x,y)  =  \sqrt{\frac{\det \Sigma_{mn}}{\det \Sigma}}  e^{-h_{nm}/2}   =  \frac{e^{-h_{nm}/2}}{\sqrt{nm(1-\rho^2) + n +m + 1}}\ .
\end{align*}
With this computation, we obtain
	\begin{equation}\label{sca_alpha}
		\begin{aligned}
			\E[ \langle \mu_n , \mu_m\rangle_\alpha]&= 2\pi \sqrt{nm} e^{-h_{nm}/2}\int_{(\S^d)^2} \frac{\| x- y\|^{\alpha-d-1}}{\sqrt{nm(1-\rho^2) + n +m + 1}}\di x \di y  \\
			& = C e^{-h_{nm}/2}\int_{-1}^1 \frac{(1-r)^{(\alpha-d-1)/2}}{ \sqrt{1-\kappa(r)^2 + \frac{1}{n} + \frac{1}{m} + \frac{1}{nm}}} (1-r^2)^{d/2-1}\di r \\
			&\leq C  \int_{-1}^1 \frac{(1-r)^{(\alpha-3)/2} (1+r)^{d/2-1}}{ (1-\kappa(r)^2)^{1/2}} \di r\ .
		\end{aligned}
	\end{equation}
	where $C$ is a constant independent of  $n$ and $m$  and  in the last inequality we use $h_{nm}\geq 0$.  Schoenberg's theorem implies that $|\kappa(r)| = 1 $ is possible only for $r\in\{-1,1\}$, so the previous integral may diverge only for $r$ near $\pm 1 $.
	\begin{itemize}
		\item
		Using the definition of CRI, as $r\to 1^-$, we have
		$$ 1- \kappa(r) = c_1 (1-r)^\beta + o((1-r)^\beta) \ ,$$
		and thus
		$$\frac{(1-r)^{(\alpha-3)/2} (1+r)^{d/2-1}}{ (1-\kappa(r)^2)^{1/2}}  \sim C (1-r)^{-1+(\alpha-1-\beta)/2 } \ . $$
Hence the integral (near $1$) is finite if and only if $\alpha-1-\beta>0$, i.e., $\gamma= d+1-\alpha<d-\beta$.
		\item For $r$ near $-1$, we consider three different cases:
		\begin{itemize}
			\item  If $\kappa(-1)\neq \pm 1$, then  the integral (near $-1$) is finite.
			\item If $\kappa(-1)= 1$, then as $r\to (-1)^+$ we have
			$$ 1- \kappa(r) = c_{-1} (1+r)^\beta + o((1+r)^\beta) \ , $$
			and thus
			$$\frac{(1-r)^{(\alpha-3)/2} (1+r)^{d/2-1}}{ (1-\kappa(r)^2)^{1/2}}  \sim C (1+r)^{-1+(d-\beta)/2 } \ . $$
			Hence, since $\beta\in(0,1]$ and $d\geq 2$, the integral (near $-1$) is finite.
			\item If $\kappa(-1) = -1$, then as $r\to (-1)^+$ we have
			$$ 1+  \kappa(r) = c_{-1} (1+r)^\beta + o((1+r)^\beta) \ , $$
			and we obtain the same conclusion as in the previous case.
		\end{itemize}
	\end{itemize}
In particular, we see that taking $n=m$ and $\alpha > \beta+1$, we have  $\E[\|\mu_n\|^2_\alpha]<\infty$. Therefore  $\|\mu_n\|_\alpha<\infty$ almost surely, that is  $\mu_n\in \mathcal{E}_\alpha^+$ a.s..
	Let us now prove that the sequence $(\mu_n)_n$ is a.s. Cauchy. Using~\eqref{sca_alpha},
	\begin{align*}\E[\| \mu_n - \mu_m\|_\alpha^2]  = C  \int_{-1}^1 (1-r)^{(\alpha-3)/2}(1+r)^{d/2-1} f_{nm}(r) \di r \ ,
	\end{align*}
	where
	\begin{align*}
		f_{nm}(r) = & \frac{e^{-h_{nn}/2}}{\sqrt{1-\kappa(r)^2 + \frac{2}{n}  + \frac{1}{n^2}}} + \frac{e^{-h_{mm}/2}}{\sqrt{1-\kappa(r)^2 + \frac{2}{m}+ \frac{1}{m^2}}}  -\frac{2e^{-h_{nm}/2}}{\sqrt{1-\kappa(r)^2 + \frac{1}{m}+ \frac{1}{n}  + \frac{1}{mn}}} \ .
	\end{align*}
	We note that $f_{nm}(r) \to 0 $ as $n,m \to + \infty $. Moreover,
	$$|f_{nm}(r)|\leq \frac{4}{\sqrt{1-\kappa(r)^2}} \ . $$
	Since we have proved that
	$$ \int_{-1}^1 \frac{(1-r)^{(\alpha-3)/2}(1+r)^{d/2-1} }{\sqrt{1-\kappa(r)^2}}\di r <\infty \ ,$$
	using the dominated convergence theorem we get $\lim_{n,m\to \infty}\E[\|\mu_n - \mu_m\|_\alpha] =0$, and hence $(\mu_n)_n$ is  a Cauchy sequence in $\mathcal{E}_\alpha^+$ on an event $\Omega_0$ of full probability.
	Let $\mu$ be the limiting (random) measure.
	We claim that there exists a set $B\subseteq \Omega_0$ with positive probability such that  $\mu(T^{-1}(u))>0$ on $B$.

	First we prove that $\mu_n(\S^d)>0$ with positive probability. To do this, we use the Paley–Zygmund inequality, i.e., for any random variable $X\geq 0$ and for any $\theta\in [0,1]$,
	\begin{equation}\label{Pal-Zy} \mathbb P( X >\theta \E[X])\geq  (1-\theta)^2\frac{\E[X]^2}{\E[X^2]} \ .
	\end{equation}
	Now,
	\begin{align*} \E[\mu_n(\S^d)] &= \E\Big[ \int_{\S^d} \sqrt{2\pi n} e^{-\frac{n(u-T(x))^2}2}\di x \Big] = \int_{\S^d} \sqrt{2\pi n} \int_{\R} \frac{1}{\sqrt{2\pi}}
		e^{-\frac{n(u - w)
				^2}{2}} e^{-\frac{w^2}2} \di w \di x\\
		&= \sqrt{n}\w_d e^{-\frac{n}{2(n+1)}u^2}\int_\R e^{-\frac{n+1}2\big(w - \frac{n}{n+1}u\big)^2} \di w
		\\
		&=\sqrt{2\pi}\w_d e^{-\frac{n}{2(n+1)}u^2}\sqrt{\frac{n}{n+1}} \ .
	\end{align*}
Since $\mu_n \in \mathcal{E}_{d+1}^+$, we also have
$\E[\mu_n(\S^d)^2] \leq C$. Using \eqref{Pal-Zy}, we obtain
	$$ \mathbb P(\mu_n(\S^d)>\theta \E[\mu_n(\S^d)])\geq \frac{(1-\theta)^2}{C} 2\pi \w_d^2 e^{-\frac{n}{2(n+1)}u^2} \frac{n}{n+1}\geq  
	p  $$
	where $p$ is a positive constant independent of $n$.
	On the other hand,
	$$\{ \mu_n(\S^d) >\theta \E[ \mu_n(\S^d)]\}
	=\Big\{ \mu_n(\S^d) >\theta \sqrt{2\pi} \w_d e^{-\frac{n}{2(n+1)}u^2} \sqrt{\frac{n}{n+1}}\Big\}
	\subseteq \{ \mu_n(\S^d) >q \}  , $$
	for  $q>0$ independent of $n$. Thus, there exists $B\subseteq \Omega_0$ such that $\mathbb P(B) \geq p$ and $\mu_n(\S^d)> q$ on $B$.

	Finally, let us prove that  $\mu(T^{-1}(u)) >0$ with positive probability. We will even prove that  $\mu$ has support in $T^{-1}(u)$ almost surely. Let $U$ be an open neighbourhood of $T^{-1}(u)$. Then $\S^d\setminus U$ is a closed set, and there exists
	$$ \delta = \min_{t\in \S^d\setminus U} |T(t) -u |> 0\, . $$
	Thus
	$$\phi_n(T(x)) \leq \sqrt{2\pi n}e^{-\frac{n\delta^2}2} \leq \sqrt{2 \pi} \delta^{-1} \mathrm e^{-1/2} \qquad \forall x \in \S^d\setminus U \, .$$ On the other hand, $\phi_n(t) \to 0 $ for any $t\neq u$,  and hence, by dominated convergence,
	$$ \mu(\S^d\setminus U) =\lim_{n\to +\infty} \mu_n(\S^d\setminus U)  = \lim_{n\to + \infty}\int_{\S^d\setminus U} \phi_n(T(x)) \di \mu_n (x)=0 \ ,$$
	which concludes the proof.
\end{proof}

\end{document}